\documentclass[11pt]{amsart}
\usepackage[T1]{fontenc}
\usepackage{times,amsmath,amssymb,amsthm,mathtools,microtype}
\usepackage[text={450pt,575pt},headheight=9pt,centering]{geometry}
\usepackage[colorlinks=true,linkcolor=blue,citecolor=blue,urlcolor=blue]{hyperref}
\newtheorem{theorem}{Theorem}[section]
\newtheorem{proposition}[theorem]{Proposition}
\newtheorem{lemma}[theorem]{Lemma}
\newtheorem{corollary}[theorem]{Corollary}
\theoremstyle{definition}\newtheorem{definition}[theorem]{Definition}
\newtheorem{example}[theorem]{Example}
\theoremstyle{remark}\newtheorem{remark}[theorem]{Remark}
\newtheorem*{remark*}{Remark}
\newcommand{\Q}{\mathbb Q}
\newcommand{\C}{\mathbb C}\newcommand{\Z}{\mathbb Z}\newcommand{\N}{\mathbb N}
\newcommand{\LL}{\mathbb L}\newcommand{\Spec}{\operatorname{Spec}}
\newcommand{\Def}{\operatorname{Def}}
\newcommand{\LM}{\operatorname{LM}}\newcommand{\lcm}{\operatorname{lcm}}
\title[Minkowski decompositions and equivariant deformations]{Minkowski decompositions and universal equivariant deformations of toric pairs}
\author{Matej Filip}
\address{University of Ljubljana, Institute of Mathematics, Physics and Mechanics, Trzaska cesta 25, Ljubljana, Slovenia}
\email{matej.filip@fe.uni-lj.si}
\thanks{Supported by Slovenian Research Agency program P1-0222 and grants J1-60011, J1-70017}
\date{}
\hypersetup{pdfauthor={Matej Filip},pdftitle={Minkowski decompositions and universal equivariant deformations of toric pairs}}
\begin{document}
\begin{abstract}
Let $X_\sigma$ be the affine normal toric variety associated with a
full-dimensional strongly convex rational polyhedral cone $\sigma$,
and let $m$ be a primitive degree with slice
$P_m=\sigma\cap[m=1]$. We construct an explicit flat algebraic
family whose completion is universal for equivariant deformations
in all degrees $-jm$, $j\ge1$, simultaneously.
We prove that the irreducible components of the reduced base
correspond bijectively to maximal lattice-friendly Minkowski
decompositions of $P_m$, and describe the induced family on each
component. When $m\in\sigma^\vee$, we also obtain a formally
universal equivariant family for the pair
$(X_\sigma,V(\chi^m))$.
These results extend earlier miniversality and component theorems
to arbitrary affine normal toric varieties and arbitrary primitive degrees.
\end{abstract}
\maketitle

\section{Introduction}

Constructing miniversal deformations and understanding the geometry
of their base spaces are central problems in algebraic geometry.
These constructions connect the deformation theory of singularities
with the local geometry of moduli spaces; see, for example,
\cite{Grauert,ArtinVersal,KSB,StevensCyclic,deJongVanStraten,Altmann}.

Two fundamental works on the construction of miniversal deformations
of toric varieties are \cite{KSB} and \cite{Altmann}.
Koll\'ar and Shepherd-Barron \cite{KSB} identified the reduced versal
components of quotient surface singularities with $P$-resolutions.
For isolated toric Gorenstein singularities $X$,
Altmann \cite{Altmann} constructed the full miniversal
deformation and identified the components of its reduced base with
maximal lattice Minkowski decompositions of the defining polytope.
He proved flatness by lifting relations among the defining equations
and established miniversality by showing that the Kodaira--Spencer
map is bijective and the obstruction map is injective.
A key feature of this setting is that, in dimension at least three,
the entire first-order deformation space is concentrated in degree
$-R^*$, so $T_X^1=T_X^1(-R^*)$, where $R^*$ is the Gorenstein
degree.

Altmann's results have been generalized in several directions.
Altmann--Kastner
\cite[Theorem 6.1]{AltmannKastner} removed the Gorenstein assumption
and extended the construction from degree $-R^*$ to degree $-m$
for an arbitrary primitive degree $m$ in the dual cone,
while retaining smoothness in codimension two.
For Gorenstein varieties, Filip
\cite[Theorem 1.1]{FilipPair} constructed miniversal deformations of
the pair consisting of the variety and its toric boundary in all
negative multiples of a degree satisfying certain restrictions, and
classified their reduced components by maximal Minkowski decompositions.

Altmann--Constantinescu--Filip \cite{ACFpoly,ACF} pursued the idea of
deriving miniversality from a universal property, with the goal of
constructing homogeneous deformations from the slice
$P_m=\sigma\cap[m=1]$, where $\sigma$ is the defining cone and
$m$ is a primitive degree. This led to the
construction of an initial object in a category of extensions of
pairs of semigroups \cite[Theorem 9.2]{ACFpoly}. That object underlies
the geometric family in \cite{ACF}.
The miniversality result of \cite[Theorem~6.2]{ACF} requires
a quite restrictive combinatorial assumption on the chosen degree,
which in particular forces $T_X^1(-jm)=0$ for all $j\ge2$
\cite[Remark~6.4]{ACF}.
Its proof still relies on lengthy computations establishing
injectivity of the obstruction map.

In this paper we remove that assumption and treat equivariant
deformations in all degrees $-jm$, $j\ge1$, simultaneously,
for an arbitrary primitive degree $m$. Our aim is to determine
both the relations among the deformation parameters and the
irreducible components of the resulting base. This leads to two
complementary constructions: explicit equations for a family whose
completion is universal (as we will see below, our approach uses
different techniques from those of the preceding works, avoiding relation-lifting
arguments and lengthy computations of the obstruction map),
and a reconstruction of Minkowski decompositions
from its fibres that proves the component classification.

Let $\LL$ be the character group, and let
$\sigma\subset\LL^*_{\mathbb R}$ be a full-dimensional strongly
convex rational polyhedral cone. Set
\[
 S=\sigma^\vee\cap\LL,\qquad A=\C[S],\qquad X=X_\sigma=\Spec A.
\]
Let $\operatorname{Art}$ be the category of local Artinian
$\C$-algebras with residue field $\C$ and local $\C$-algebra
homomorphisms. For a primitive degree $m\in\LL$, we consider
the following deformation functors
\[
 \Def^X:\operatorname{Art}\longrightarrow\mathrm{Sets},\qquad
 \Def_m^X:\operatorname{Art}_m\longrightarrow\mathrm{Sets}.
\]
Here $\operatorname{Art}_m$ consists of local Artinian
$\LL$-graded $\C$-algebras
$R=\C\oplus\bigoplus_{j\ge1}R_{jm}$ with graded local morphisms.
The functor $\Def^X$ classifies flat deformations of $X$.
For $\Def_m^X$, the deformation algebra is $\LL$-graded, and
all structure maps and isomorphisms preserve the grading.
We call $\Def_m^X$ the \mbox{(torus-)}equivariant deformation functor.
When $m\in S$, we also consider the corresponding pair functors
$\Def^{(X,V(\chi^m))}$ and $\Def_m^{(X,V(\chi^m))}$;
in the equivariant case, the divisor ideal is homogeneous.
For the precise graded definitions, cf.~Definition~\ref{def:pair}.

When $m\in S$, we first construct a flat family of the
pair $(X_\sigma,V(\chi^m))$ directly from the toric ideal.
In Section~\ref{sec:finite}, we perturb a monic Gr\"obner basis
homogeneously and obtain the parameter algebra by setting the
coefficients of the $S$-polynomial remainders to zero. Polynomial
division and Buchberger's criterion show that the standard monomials
form a basis over this algebra, proving flatness.
Section~\ref{sec:obstruction} establishes equivariant universality:
the adapted presentation makes the relevant homogeneous generator
lifts unique, so every pair deformation determines a unique
classifying map. After forgetting the divisor and removing the
parameter that records its choice, we obtain a formally universal
equivariant family of $X_\sigma$. The same section compares these families
with hulls of $\Def^X$ and $\Def^{(X,V(\chi^m))}$.
Here a hull is a miniversal formal deformation over a complete local
Noetherian base; see Subsection~\ref{subsec:usual-miniversal} for the
precise definition. Whenever such a hull exists, its formal family
can be chosen torus-equivariantly so that its restriction to the
prescribed homogeneous parameters is isomorphic to the corresponding
universal equivariant family.

Section~\ref{sec:components} identifies the irreducible components
of the reduced pair base with maximal lattice-friendly Minkowski
decompositions of $P_m$, in the sense of \cite{ACFpoly}, and
describes the family on each component. The key converse recovers
a lattice-friendly decomposition from the multiplication in a
fibre, showing that these decompositions exhaust the reduced
components.

Finally, Section~\ref{sec:arbitrary-degrees} identifies the equivariant
deformation functors for $\sigma$ and $\sigma\cap[m\ge0]$.
This isomorphism extends the
universal equivariant family of $X_\sigma$ and its component
correspondence to arbitrary primitive degrees.

The following theorems are the main results of this paper.

\begin{theorem}\label{thm:prorepresentability}
Suppose $m\in S$. There exist a finitely generated positively
graded $\C$-algebra $B_0$ and an explicit deformation family of
$(X,V(\chi^m))$ whose family algebra $\mathcal A_0$ is free over $B_0$.
Its completion is formally universal for
$\Def_m^{(X,V(\chi^m))}$.

The irreducible components of the reduced universal base
$(\Spec\widehat B_0)_{\mathrm{red}}$ are in one-to-one correspondence
with maximal lattice-friendly Minkowski decompositions of $P_m$.
A maximal decomposition $\mathcal D$ into $k>0$ nontrivial summands
gives a component of dimension $k$.
\end{theorem}

\begin{theorem}\label{thm:main}
For $m\notin-S$, there is an explicit finitely generated positively graded
$\C$-algebra $B$, with parameter degrees in $\Z_{>0}m$,
and a finitely generated $B$-algebra $\mathcal A$ that is free as a
$B$-module and has special fibre $A$. The completion of this family
at the origin is formally universal for $\Def_m^X$.

The irreducible components of the reduced universal equivariant
base $(\Spec\widehat B)_{\mathrm{red}}$ are in one-to-one
correspondence with maximal lattice-friendly Minkowski
decompositions of $P_m$.
For a maximal decomposition $\mathcal D$ into $k>0$ nontrivial
summands, the corresponding component has dimension $k-1$.
\end{theorem}

\begin{theorem}\label{thm:usual-miniversality}
If $\Def^X$ admits a hull, then a corresponding miniversal formal
family can be chosen torus-equivariantly so that its restriction to
homogeneous parameters of degrees $jm$, $j\ge1$, is isomorphic
to the universal equivariant family over $\widehat B$.
For $m\in S$, the same statement holds for
$\Def^{(X,V(\chi^m))}$, with $\widehat B$ replaced by $\widehat B_0$.
\end{theorem}

For Gorenstein $X$, the equality $V(\chi^{R^*})=\partial X$ shows
that Theorem~\ref{thm:prorepresentability} gives a universal formal
equivariant deformation of $(X,\partial X)$, hence miniversality for
the pair in all negative multiples of $R^*$.
Theorem~\ref{thm:main} gives the universal equivariant deformation
of $X$ in the same degrees. For isolated Gorenstein $X$ of dimension at least three,
this is a miniversal formal deformation for $\Def^X$ by
Theorem~\ref{thm:usual-miniversality}, since the entire first-order
deformation space is concentrated in degree $-R^*$.

These results connect algebraic geometry with discrete mathematics
by describing the reduced components of deformation spaces and their
flat families in terms of lattice-friendly Minkowski decompositions.

\noindent\textbf{Acknowledgments.}
I thank Klaus Altmann and Alexandru Constantinescu for many helpful
discussions.

\section{The construction of a deformation family}\label{sec:finite}

In this section we construct a flat family of the pair
$(X,V(\chi^m))$ for a primitive character $m\in S$, where
$X=X_\sigma=\Spec\C[S]$, $S=\sigma^\vee\cap\LL$, and $A=\C[S]$.

\begin{definition}\label{def:pair-family}
Let $B$ be a finitely generated $\C$-algebra with an augmentation
$B\to\C$. A \emph{flat family} of $(X,V(\chi^m))$ over $\Spec B$
consists of a morphism
\[
 (\Spec\mathcal A,V(\mathcal I))\longrightarrow\Spec B,
\]
where $\mathcal A$ is a finitely generated $B$-algebra and
$\mathcal I\subset\mathcal A$ is an ideal such that both
$\mathcal A$ and $\mathcal A/\mathcal I$ are $B$-flat, together
with an isomorphism $\mathcal A\otimes_B\C\simeq A$ carrying
the image of $\mathcal I$ to $(\chi^m)$.
\end{definition}

Choose an integral $\ell\in\operatorname{int}\sigma$, so
$\ell(s)>0$ for every nonzero $s\in S$, in particular for $s=m$.
Let $s_1,\ldots,s_n$ be the minimal generators of $S$ different
from $m$, and use the presentation
\[
 P=\C[U,x_1,\ldots,x_n]\twoheadrightarrow A,\qquad
 U\longmapsto\chi^m,\quad x_i\longmapsto\chi^{s_i},
 \qquad \deg U=m,\quad\deg x_i=s_i.
\]
Write $I_S$ for the kernel. Order monomials first by their total
$\ell$-weight, then by the $\ell$-weight contributed by the variables
other than $U$, and finally by a fixed monomial order.
Recall that a monomial order is a total well-order compatible with
multiplication. For a nonzero polynomial $f$, its leading monomial
$\LM(f)$ is the largest monomial occurring in $f$, without its coefficient.

Let $g_1,\ldots,g_r$ be the reduced monic homogeneous Gröbner basis
of $I_S$ for the chosen monomial order. Here the terminology is
as follows: the Gröbner-basis property means that the leading
monomial of every nonzero $f\in I_S$ is divisible by some
$\LM(g_\alpha)$. Monic means that each $g_\alpha$ has leading
coefficient $1$, and reduced means that no monomial of $g_\alpha$
is divisible by $\LM(g_\beta)$ for $\beta\ne\alpha$.
Homogeneous means that all terms have the same character degree.

Let $\mathcal N$ be the set of standard monomials in
$P=\C[U,x_1,\ldots,x_n]$, namely those divisible by none of the
$\LM(g_\alpha)$. Their residue classes form a homogeneous
$\C$-basis of $A$, so every $s\in S$ has a standard monomial
representative. If distinct standard monomials $\mu,\nu$ had
degree $s$, then $\mu-\nu\in I_S$ would have a standard leading
monomial, contradicting the Gröbner-basis property. Denote the
unique standard monomial of degree $s$ by $N_s$.

For every nonzero $a\in S$, define
\begin{equation}\label{eq:boundarydecomp}
 \partial_{\N a}S=\{b\in S:b-a\notin S\},\qquad
 S=\coprod_{b\in\partial_{\N a}S}(b+\N a).
\end{equation}
Indeed, for $s\in S$, the largest integer $q\ge0$ with
$s-qa\in S$ is finite and determines the unique expression $s=b+qa$ with
$b\in\partial_{\N a}S$.

For the chosen monomial order, no $\LM(g_\alpha)$ is divisible by $U$.
Suppose $U\mid\LM(g_\alpha)$. By homogeneity, all terms have
the same total $\ell$-weight. A term without $U$ would therefore
have greater weight in the $x_i$ and would be larger than
$\LM(g_\alpha)$ in the chosen order. Thus $U$ divides every term.
Since $I_S$ is prime and $U\notin I_S$, the quotient
$g_\alpha/U$ lies in $I_S$. The Gr\"obner-basis property then
gives a basis leading monomial dividing $\LM(g_\alpha)/U$.
It would properly divide $\LM(g_\alpha)$, contradicting reducedness.
Consequently, every $U^qN_b$ with $b\in\partial_{\N m}S$ is
standard. By \eqref{eq:boundarydecomp} and uniqueness of the
standard monomial in each degree, these are all the standard
monomials, and $N_{b+qm}=U^qN_b$.

For $1\le\alpha\le r$, put $d_\alpha=\deg g_\alpha\in S$.
Applying \eqref{eq:boundarydecomp} with $a=m$, define
\[
 \eta_{\Z,\alpha}
 :=\max\{q\in\Z_{\ge0}:d_\alpha-qm\in S\},\qquad
 b_\alpha=d_\alpha-\eta_{\Z,\alpha}m\in\partial_{\N m}S.
\]
The boundary decomposition proves that this maximum is finite.
For $j\ge0$, the correction degree $d_\alpha-jm$ belongs to $S$
exactly when $j\le\eta_{\Z,\alpha}$, since then it equals
$b_\alpha+(\eta_{\Z,\alpha}-j)m$.
With $W_\alpha=N_{b_\alpha}$, the standard-monomial identity above gives
\[
 N_{d_\alpha-jm}=U^{\eta_{\Z,\alpha}-j}W_\alpha,
 \qquad 0\le j\le\eta_{\Z,\alpha}.
\]
Reducedness and homogeneity force every nonleading monomial of
$g_\alpha$ to be the unique standard monomial $N_{d_\alpha}$.
Since $g_\alpha\in I_S$ is monic, we have
$g_\alpha=\LM(g_\alpha)-N_{d_\alpha}$.

For $1\le j\le\eta_{\Z,\alpha}$, introduce a deformation parameter
$t_{\alpha,j}$ of degree $jm$, and define
\begin{equation}\label{eq:G}
 G_\alpha=g_\alpha-\sum_{j=1}^{\eta_{\Z,\alpha}}t_{\alpha,j}N_{d_\alpha-jm},
 \qquad T=\C[t_{\alpha,j}].
\end{equation}
Substituting these expressions into \eqref{eq:G} and factoring out
$W_\alpha$ gives
\begin{equation}\label{eq:finiteeta}
 G_\alpha=\LM(g_\alpha)-p_\alpha(U)W_\alpha,\qquad
 p_\alpha(U)=U^{\eta_{\Z,\alpha}}+\sum_{j=1}^{\eta_{\Z,\alpha}}
                 t_{\alpha,j}U^{\eta_{\Z,\alpha}-j}.
\end{equation}
If $\eta_{\Z,\alpha}=0$, no parameter occurs in $G_\alpha$.
Each correction term has character degree $d_\alpha$, so
$G_\alpha$ is homogeneous. Treating the parameters as coefficients,
its correction monomials have $\ell$-weight
$\ell(d_\alpha)-j\ell(m)<\ell(d_\alpha)$.
Hence $\LM(G_\alpha)=\LM(g_\alpha)$, and $G_\alpha$ remains monic.

We obtain equations in the parameters by requiring every pairwise
$S$-polynomial to have zero remainder upon division by
$(G_1,\ldots,G_r)$.
For each pair $\alpha<\beta$, form the $S$-polynomial
\[
 S_{\alpha\beta}
 =\frac{\mu_{\alpha\beta}}{\LM(g_\alpha)}G_\alpha
  -\frac{\mu_{\alpha\beta}}{\LM(g_\beta)}G_\beta,
 \qquad
 \mu_{\alpha\beta}=\lcm\bigl(\LM(g_\alpha),\LM(g_\beta)\bigr).
\]
Here $\lcm$ takes the coordinatewise maximum of the monomial
exponents. The two multiples have the same leading monomial
$\mu_{\alpha\beta}$, which cancels in their difference.
We now divide $S_{\alpha\beta}$ by the ordered list
$(G_1,\ldots,G_r)$, treating the parameters as coefficients
(we call the following computation the \emph{division procedure}).
Start with the polynomial $H=S_{\alpha\beta}$ and remainder
$R_{\alpha\beta}=0$. As long as $H\ne0$, write its leading term
as $c\mu$, where $\mu=\LM(H)$ is a monomial in $U,x_1,\ldots,x_n$ and
$0\ne c\in T$ is its coefficient, a polynomial in the deformation
parameters. There are two cases:
\begin{itemize}
\item If some $\LM(G_\gamma)$ divides $\mu$, choose the smallest
such index $\gamma$ and replace
\[
 H\longmapsto H-c\frac{\mu}{\LM(G_\gamma)}G_\gamma.
\]
The subtracted polynomial has leading term $c\mu$, so this term
cancels and only smaller monomials are introduced.
\item If no $\LM(G_\gamma)$ divides $\mu$, then $\mu$ is standard.
Add $c\mu$ to $R_{\alpha\beta}$ and replace $H$ by $H-c\mu$.
\end{itemize}
In either case, the leading monomial strictly decreases whenever
the new $H$ is nonzero. The process therefore terminates, since
the monomial order is a well-order. Monicity ensures that cancelling
$c\mu$ requires no inversion of coefficients in $T$.

At termination, the remainder is a finite sum
\[
 R_{\alpha\beta}
 =\sum_{N\in\mathcal N}\rho_{\alpha\beta,N}(t)N,
 \qquad \rho_{\alpha\beta,N}(t)\in T.
\]
Choosing the smallest index $\gamma$ with $\LM(G_\gamma)\mid \mu$
at each division step makes the procedure deterministic and fixes
the remainder. Let $J_{\mathrm{coeff}}\subset T$ be the ideal generated
by all coefficient polynomials $\rho_{\alpha\beta,N}(t)$, for all
pairs $\alpha<\beta$. These are precisely the equations requiring
every remainder to vanish. Set
\begin{equation}\label{eq:B0}
 B_0=T/J_{\mathrm{coeff}},\qquad
 \mathcal A_0=B_0[U,x_1,\ldots,x_n]/(G_1,\ldots,G_r).
\end{equation}
The division identities descend from $T$ to $B_0$, where every
remainder $R_{\alpha\beta}$ vanishes. Buchberger's criterion
therefore shows that the images of $G_1,\ldots,G_r$ form a
Gröbner basis in $B_0[U,x_1,\ldots,x_n]$; see e.g.\
\cite[Chapter 4]{AL}. The monic form of the criterion applies
here because the divisions and the leading-term cancellation
argument require no inversion of coefficients, even when $B_0$
has zero divisors or nilpotents.
We also write $U$ for its image in $\mathcal A_0$, and denote by
$V(U)=\Spec(\mathcal A_0/(U))$ the closed subscheme it defines.
The following proposition shows that the family algebra
$\mathcal A_0$, together with the divisor ideal $(U)$, gives the
pair family of Theorem~\ref{thm:prorepresentability}.

\begin{proposition}\label{prop:free}
The algebra $\mathcal A_0$ is free over $B_0[U]$, with basis
$\{N_b:b\in\partial_{\N m}S\}$, and $\mathcal A_0/(U)$ is
free over $B_0$ with basis the images of these monomials.
Consequently,
\[
 (\Spec\mathcal A_0,V(U))\longrightarrow\Spec B_0
\]
is a flat family of $(X,V(\chi^m))$.
\end{proposition}
\begin{proof}
Division expresses every residue class as a $B_0$-linear
combination of standard monomials, proving spanning. For linear
independence, suppose a finite sum $\sum_{N\in\mathcal N}r_NN$,
with $r_N\in B_0$, belongs to $(G_1,\ldots,G_r)$. If some
$r_N\ne0$, this is a nonzero polynomial with a standard leading
monomial. The Gröbner-basis property would make that monomial
divisible by some $\LM(G_\alpha)=\LM(g_\alpha)$, a contradiction.
Thus all $r_N$ vanish.
The standard monomials are $U^qN_b$, with
$b\in\partial_{\N m}S$ and $q\ge0$, so this basis gives
\begin{equation}\label{eq:adaptedfree}
 \mathcal A_0=\bigoplus_{b\in\partial_{\N m}S}B_0[U]N_b,
 \qquad
 \mathcal A_0/(U)=\bigoplus_{b\in\partial_{\N m}S}B_0\overline{N_b},
\end{equation}
where the bars denote images in the quotient by $U$.
Both algebras are therefore $B_0$-flat. At the origin, the equations
$G_\alpha$ specialize to $g_\alpha$ and $U$ maps to $\chi^m$,
so their special fibres are $A$ and $A/(\chi^m)$, respectively.
\end{proof}

\begin{example}\label{ex:eighth}
Let $\sigma=\mathbb R_{\ge0}^2\subset\LL^*_{\mathbb R}$ be the
positive quadrant in the coordinates dual to $(p,q)$, where
\[
 \LL=\{(p,q)\in\Z^2:p+3q\in8\Z\}.
\]
Then $S=\sigma^\vee\cap\LL=\LL\cap\Z_{\ge0}^2$, and
$X=\Spec\C[S]$ is the cyclic quotient singularity $\frac18(1,3)$.
Take $m=(2,2)$. The minimal generators different from $m$ are
\[
 s_1=(8,0),\qquad s_2=(5,1),\qquad
 s_3=(1,5),\qquad s_4=(0,8).
\]
Use the presentation $P=\C[U,z_0,z_1,z_3,z_4]\twoheadrightarrow A$
sending $U$ to $\chi^m$ and $z_0,z_1,z_3,z_4$ to
$\chi^{s_1},\ldots,\chi^{s_4}$, respectively; thus $\deg U=m$.
The vectors $(3,-1)$ and $m$ form a basis of $\LL$, so $m$ is
primitive. In the basis dual to $((3,-1),m)$, the primitive ray
generators of $\sigma$ are $(-1,2)$ and $(3,2)$.
Thus the slice $P_m=\sigma\cap[m=1]$ identifies with
$[-\tfrac12,\tfrac32]$.

Take $\ell(p,q)=p+q$, so $\ell(m)=4$. As in the construction
above, first compare total $\ell$-weights. For equal weights,
compare the weights
contributed by the variables other than $U$, namely $(8,6,0,6,8)$ on
$(z_0,z_1,U,z_3,z_4)$. If these weights also agree, compare total
degrees in $z_0,z_1,z_3,z_4$. If the degrees are equal as well, use
lexicographic order with
$z_0>z_1>z_3>z_4>U$.
For this order, the reduced monic Gr\"obner basis is
\begin{equation}\label{eq:eighth-central}
\begin{aligned}
 g_1&=z_1^2-Uz_0,&g_2&=z_3^2-Uz_4,\\
 g_3&=z_0z_3-U^2z_1,&g_4&=z_1z_4-U^2z_3,\\
 g_5&=z_1z_3-U^3,&g_6&=z_0z_4-U^4.
\end{aligned}
\end{equation}
The first monomial of each $g_\alpha$ is its leading monomial.
For instance, $z_1^2$
and $Uz_0$ both have $\ell$-weight $12$, while their secondary
weights are $12$ and $8$. The $S$-polynomial for $g_1,g_5$ is
\[
 z_3g_1-z_1g_5=-Ug_3.
\]
Its leading term is $-Uz_0z_3$, so the division procedure adds
$Ug_3$, leaving zero. All remaining $S$-polynomials also have
zero remainder. A point $(p,q)\in S$ belongs to $\partial_{\N m}S$ precisely
when $p<2$ or $q<2$. Hence
\[
 \partial_{\N m}S
 =\{(8n,0),(0,8n),(8n+5,1),(1,8n+5):n\ge0\}.
\]
The corresponding standard monomials form the set
\[
 \mathcal W=\{N_b:b\in\partial_{\N m}S\}
 =\{1\}\cup\{z_0^n,z_4^n:n\ge1\}
       \cup\{z_1z_0^n,z_3z_4^n:n\ge0\}.
\]
All standard monomials are $U^qW$, with $q\ge0$ and
$W\in\mathcal W$. The equation degrees, parameter bounds, and
monomials in \eqref{eq:finiteeta} are
\[
\begin{array}{c|c|c|c|c}
 \alpha&d_\alpha&\eta_{\Z,\alpha}&b_\alpha&W_\alpha\\ \hline
 1&(10,2)&1&(8,0)&z_0\\
 2&(2,10)&1&(0,8)&z_4\\
 3&(9,5)&2&(5,1)&z_1\\
 4&(5,9)&2&(1,5)&z_3\\
 5&(6,6)&3&(0,0)&1\\
 6&(8,8)&4&(0,0)&1
\end{array}
\]
For example, $d_1-m=(8,0)$ gives $N_{d_1-m}=z_0$, whereas
$d_1-2m=(6,-2)\notin S$, so $\eta_{\Z,1}=1$.
Each correction monomial has $\ell$-weight $\ell(d_\alpha)-4j$,
so $\LM(G_\alpha)=\LM(g_\alpha)$, as in the general construction.
Substitution into \eqref{eq:G} gives the six lifted equations
\begin{equation}\label{eq:eighth-coefficients}
\begin{aligned}
 G_1&=z_1^2-(U+t_{1,1})z_0,\\
 G_2&=z_3^2-(U+t_{2,1})z_4,\\
 G_3&=z_0z_3-(U^2+t_{3,1}U+t_{3,2})z_1,\\
 G_4&=z_1z_4-(U^2+t_{4,1}U+t_{4,2})z_3,\\
 G_5&=z_1z_3-(U^3+t_{5,1}U^2+t_{5,2}U+t_{5,3}),\\
 G_6&=z_0z_4-(U^4+t_{6,1}U^3+t_{6,2}U^2+t_{6,3}U+t_{6,4}).
\end{aligned}
\end{equation}
There are thirteen parameters $t_{\alpha,j}$, with $\deg t_{\alpha,j}=jm$.
As in \eqref{eq:finiteeta}, write
$G_\alpha=\LM(g_\alpha)-p_\alpha(U)W_\alpha$; thus
$p_1(U)=U+t_{1,1}$ and
$p_3(U)=U^2+t_{3,1}U+t_{3,2}$.
For $G_1,G_5$, form $S_{15}=z_3G_1-z_1G_5$.
Applying the division procedure to this $S$-polynomial, we first
add $UG_3$ and then $t_{1,1}G_3$, giving
\[
 S_{15}+p_1(U)G_3=\bigl(p_5-p_1p_3\bigr)z_1.
\]
Every $U^kz_1$ is standard, so the right-hand side is the
remainder $R_{15}$. Consequently, the coefficients of
$p_5-p_1p_3$ as a polynomial in $U$ belong to
$J_{\mathrm{coeff}}$.
Division of the remaining $S$-polynomials gives exactly the polynomial identities
\begin{equation}\label{eq:eighth-overlaps}
 p_5=p_1p_3=p_2p_4,\qquad p_6=p_3p_4.
\end{equation}
Indeed, besides $R_{15}$, the remainders
$R_{25}=(p_5-p_2p_4)z_3$ and
$R_{36}=(p_6-p_3p_4)z_3$ supply their coefficients, and these
identities make all fifteen pair remainders vanish.

These are identities of polynomials in $U$, so the coefficients
of each power of $U$ must agree. For example,
\[
\begin{aligned}
 p_1p_3
 &=(U+t_{1,1})(U^2+t_{3,1}U+t_{3,2})\\
 &=U^3+(t_{1,1}+t_{3,1})U^2
       +(t_{3,2}+t_{1,1}t_{3,1})U+t_{1,1}t_{3,2}.
\end{aligned}
\]
Matching the coefficients of $U^2,U,1$ with those of $p_5$
gives the first column below. The second follows from $p_5=p_2p_4$:
\[
\begin{aligned}
 t_{5,1}&=t_{1,1}+t_{3,1},
 &t_{5,1}&=t_{2,1}+t_{4,1},\\
 t_{5,2}&=t_{3,2}+t_{1,1}t_{3,1},
 &t_{5,2}&=t_{4,2}+t_{2,1}t_{4,1},\\
 t_{5,3}&=t_{1,1}t_{3,2},
 &t_{5,3}&=t_{2,1}t_{4,2}.
\end{aligned}
\]
The identity $p_6=p_3p_4$ similarly gives
\[
\begin{aligned}
 t_{6,1}&=t_{3,1}+t_{4,1},\\
 t_{6,2}&=t_{3,2}+t_{3,1}t_{4,1}+t_{4,2},\\
 t_{6,3}&=t_{3,2}t_{4,1}+t_{3,1}t_{4,2},\\
 t_{6,4}&=t_{3,2}t_{4,2}.
\end{aligned}
\]
The ideal $J_{\mathrm{coeff}}$ is generated by the differences
between the left and right sides of these ten equations.
\end{example}

\section{Universality and miniversality}\label{sec:obstruction}

In this section we prove equivariant universality for the pair $(X,V(\chi^m))$
and for $X$. In Subsection~\ref{subsec:usual-miniversal}, we pass
to $\Def^X$ and $\Def^{(X,V(\chi^m))}$ on $\operatorname{Art}$
and compare their hulls, when they exist, with these universal
equivariant families.

\begin{definition}\label{def:pair}
Let $m\in\LL$ be primitive. Let $\operatorname{Art}_m$ be the
category of local Artinian $\LL$-graded $\C$-algebras
\begin{equation}\label{eq:bases}
 R=\C\oplus\bigoplus_{j\ge1}R_{jm},\qquad
 \mathfrak m_R=\bigoplus_{j\ge1}R_{jm},
\end{equation}
with graded local $\C$-algebra homomorphisms.
The functor $\Def_m^X:\operatorname{Art}_m\to\mathrm{Sets}$
assigns to $R$ the isomorphism classes of pairs $(A_R,\iota_R)$,
where $A_R$ is a finitely generated flat commutative $\LL$-graded
$R$-algebra
and $\iota_R:A_R/\mathfrak m_RA_R\xrightarrow{\sim}A$ is a
graded $\C$-algebra isomorphism. An isomorphism
$(A_R,\iota_R)\to(A'_R,\iota'_R)$ is a graded $R$-algebra
isomorphism $\varphi:A_R\to A'_R$ satisfying
$\iota'_R\circ\overline\varphi=\iota_R$, where
$\overline\varphi:A_R/\mathfrak m_RA_R\to
A'_R/\mathfrak m_RA'_R$ is the map induced by $\varphi$ on the
special fibres.

If $m\in S$, the functor
$\Def_m^{(X,V(\chi^m))}:\operatorname{Art}_m\to\mathrm{Sets}$
assigns to $R$ the isomorphism classes of triples
$(A_R,I_R,\iota_R)$, where $(A_R,\iota_R)$ is as above and
$I_R\subset A_R$ is a homogeneous ideal such that $A_R/I_R$
is $R$-flat and
$\iota_R\bigl((I_R+\mathfrak m_RA_R)/\mathfrak m_RA_R\bigr)=(\chi^m)$.
Isomorphisms additionally satisfy $\varphi(I_R)=I'_R$.

For a base map $R\to R'$, both functors act by the base change
$A_R\mapsto A_R\otimes_RR'$, with the induced special-fibre
identification; for pairs, the ideal becomes
$I_R(A_R\otimes_RR')$. We call these the torus-equivariant
deformation functors on $\operatorname{Art}_m$.
\end{definition}

For the remainder of this section, assume $m\in S$.
Throughout Subsections~\ref{subsec:pairs} and~\ref{subsec:minimal},
we work over $\operatorname{Art}_m$, with graded deformations
and morphisms as in Definition~\ref{def:pair}.
Equivalently, the torus $\mathbb T=\Spec\C[\LL]$ acts on both
the total space and the base, and the deformation morphism is
equivariant.

Using homogeneous presentations and lifts, the standard tangent
and obstruction calculations apply degree by degree, since taking
homogeneous components is exact. Formal smoothness, versality
and universality are understood in this graded category;
we recall the corresponding forms of the standard definitions
in \cite{Sernesi} below.

For constructions of miniversal deformations by lifting equations
and relations and analysing obstruction maps, see
\cite{Altmann,deJongVanStraten}; general accounts of these methods
are given in \cite{deJongPfister,StevensBook}.

\subsection{The universal pair family}\label{sec:versal}\label{subsec:pairs}

Write
\[
 T_X^i=\operatorname{Ext}_A^i(L_{A/\C},A),\qquad i=1,2,
\]
for the cotangent cohomology groups, where $L_{A/\C}$ is the
cotangent complex. The torus action on $X$ induces the
$\LL$-grading
\[
 T_X^i=\bigoplus_{u\in\LL}T_X^i(u),\qquad i=1,2.
\]
These graded cotangent cohomology groups have been extensively
studied for affine toric varieties; see, for example,
\cite{Altmann,AltmannSletsjoe,FilipHochschild}.

Let $C$ be a complete local Noetherian $\C$-algebra with residue
field $\C$ and maximal ideal $\mathfrak m_C$, whose Artinian
quotients carry compatible gradings of the form \eqref{eq:bases}.
A formal graded family over $C$ is a compatible system of graded
flat deformations over $C/\mathfrak m_C^{n+1}$.
For $R\in\operatorname{Art}_m$, write $h_C(R)$ for the continuous
local graded $\C$-algebra maps $C\to R$. For a fixed formal family, a map $\varphi\in h_C(R)$
is a \emph{classifying map} for a deformation over $R$ if the
pullback of that family along $\varphi$ is isomorphic to the given
deformation. The algebra $C$ \emph{prorepresents} a deformation
functor $F:\operatorname{Art}_m\to\mathrm{Sets}$ if there is a
natural bijection $h_C(R)\simeq F(R)$ for every
$R\in\operatorname{Art}_m$; the corresponding formal family is
then \emph{universal}.

For $j\ge1$, put $D_j=\C[\epsilon]/(\epsilon^2)$ with
$\deg\epsilon=jm$. Taking homogeneous components in the standard
first-order classification \cite[Proposition~3.1.1]{Sernesi} gives
$\Def_m^X(D_j)\simeq T_X^1(-jm)$.
For a formal family of $X$ over $C$, pullback along $C\to D_j$
defines its \emph{Kodaira--Spencer map}
\[
 \operatorname{KS}_{C,j}:
 \operatorname{Hom}_{\C}\bigl((\mathfrak m_C/\mathfrak m_C^2)_{jm},\C\bigr)
 \longrightarrow T_X^1(-jm).
\]

A graded surjection $R'\twoheadrightarrow R$ is a
\emph{small extension} if its kernel $E$ satisfies
$\mathfrak m_{R'}E=0$. A natural transformation $F\to G$ of
functors on these bases is \emph{formally smooth} if, for every
small extension, the natural map
\[
 F(R')\longrightarrow F(R)\mathbin{\times}_{G(R)}G(R')
\]
is surjective. A family inducing $h_C\to F$ is
\emph{formally versal} if this transformation is formally smooth:
a classifying map over $R$ must lift compatibly with any prescribed
deformation over $R'$.
It is \emph{formally miniversal} if it is formally versal and
its Kodaira--Spencer map is bijective in every degree.
In particular, a prorepresenting family is formally miniversal.

The next two lemmas prepare the proof of universality.

\begin{lemma}[{see e.g.~\cite[Tag~051F]{Stacks}}]\label{lem:basis}
Let $R$ be a local Artinian ring and $M$ a flat $R$-module.
Any lifts to $M$ of a basis of $M/\mathfrak m_RM$ form an
$R$-basis of $M$.
\end{lemma}

Let $A_R$ be a graded flat deformation of $A$ over an Artinian
base $R$, with its given special-fibre identification
$A_R/\mathfrak m_RA_R\simeq A$.
We write $(A_R)_s$ for its homogeneous component of character
degree $s$.
For later use, flatness also gives the graded identifications
\begin{equation}\label{eq:flatgraded}
 \mathfrak m_R^pA_R/\mathfrak m_R^{p+1}A_R
 \simeq(\mathfrak m_R^p/\mathfrak m_R^{p+1})\otimes_\C A.
\end{equation}

\begin{lemma}\label{lem:divisor-lifts}
For every $0\ne s\in S$, let $I_R\subset A_R$ be a homogeneous
ideal such that $A_R/I_R$ is $R$-flat and $I_R$ maps onto
$(\chi^s)$ under the special-fibre map $A_R\to A$.
Then $I_R=(f)$ for a unique homogeneous element $f\in(A_R)_s$
mapping to $\chi^s$. Conversely, every such lift $f$ generates
an ideal with these properties. This correspondence commutes
with base change.
\end{lemma}
\begin{proof}
Since $A_R$ and $A_R/I_R$ are $R$-flat, the exact sequence
\[
 0\longrightarrow I_R\longrightarrow A_R\longrightarrow A_R/I_R
 \longrightarrow0
\]
shows that $I_R$ is $R$-flat. Moreover, the sequence remains exact
after reduction modulo $\mathfrak m_R$. Thus $I_R/\mathfrak m_RI_R$ identifies
with $(\chi^s)\subset A$, and we can choose a homogeneous
lift $f\in I_R$ of $\chi^s$.
Multiplication $A_R\to I_R$ by $f$ is an isomorphism after reduction.
Nilpotence of $\mathfrak m_R$ gives surjectivity; flatness of $I_R$
then makes its kernel have zero reduction, hence zero. Thus
$I_R=(f)$ and $f$ is regular.

For uniqueness, we first show that the degree-zero subalgebra
$(A_R)_0$ consists only of scalars, so $(A_R)_0=\C$.
Choose homogeneous lifts $e_b$ of $\chi^b$ for $b\in S$, with
$e_0=1$. They form an $R$-basis by Lemma~\ref{lem:basis}, so
\[
 (A_R)_0=\bigoplus_{b\in S}R_{-b}e_b=\C\cdot1.
\]
Indeed, $R_{-b}\ne0$ requires $b=-jm$ for some $j\ge0$;
since $0\ne m\in S$ and $S\cap(-S)=\{0\}$, this is possible
only for $b=0$. Now let $g\in(I_R)_s$ be another lift of $\chi^s$.
Write $g=af$ using $I_R=(f)$. Taking the degree-$s$ component
gives $g=a_0f$, where $a_0\in(A_R)_0=\C$. Reduction to the
special fibre gives $\chi^s=a_0\chi^s$, hence $a_0=1$ and $g=f$.

Conversely, let $f$ be any homogeneous lift of $\chi^s$ and
choose homogeneous lifts $e_b$ of $\chi^b$ for
$b\in\partial_{\N s}S$. By \eqref{eq:boundarydecomp}, the
elements $f^qe_b$, $q\ge0$, lift a basis of $A$ and are an
$R$-basis by Lemma~\ref{lem:basis}. Hence $A_R$ is free over
$R[f]$, and $A_R/(f)$ is $R$-free with basis the images of the
$e_b$. Uniqueness of the generator proves compatibility with base change.
\end{proof}

Put $\mathfrak m_0=(t_{\alpha,j})\subset B_0$ and denote the
completions by
\[
 \widehat B_0=\varprojlim_n B_0/\mathfrak m_0^{n+1},\qquad
 \widehat{\mathcal A}_0
 =\varprojlim_n\mathcal A_0/\mathfrak m_0^{n+1}\mathcal A_0.
\]
We write
\[
 \widehat\pi_0:
 (\operatorname{Spf}\widehat{\mathcal A}_0,V(U))
 \longrightarrow\operatorname{Spf}\widehat B_0
\]
for the formal pair family, that is, the compatible system of its
reductions over $B_0/\mathfrak m_0^{n+1}$.

\begin{proposition}\label{prop:marked-pair}\label{prop:versal}
The completion $\widehat\pi_0$ of the pair family constructed in
Section~\ref{sec:finite} induces a natural isomorphism
$h_{\widehat B_0}\simeq\Def_m^{(X,V(\chi^m))}$; thus
$\widehat B_0$ prorepresents this functor.
\end{proposition}
\begin{proof}
Consider a pair deformation $(A_R,I_R)$.
Lemma~\ref{lem:divisor-lifts} determines a unique homogeneous
generator $U_R$ of $I_R$ lifting $\chi^m$. Every other generator
degree $s_i$ is indecomposable in $S$ and differs from $m$, so
$s_i-jm\notin S$ for all $j\ge1$.
Two homogeneous lifts of $\chi^{s_i}$ differ
by an element of $(\mathfrak m_RA_R)_{s_i}$.
For every $p\ge1$, equation~\eqref{eq:flatgraded} gives
\[
 \bigl(\mathfrak m_R^pA_R/\mathfrak m_R^{p+1}A_R\bigr)_{s_i}
 \simeq\bigoplus_{j\ge1}
 (\mathfrak m_R^p/\mathfrak m_R^{p+1})_{jm}
 \otimes_\C A_{s_i-jm}=0.
\]
Since $\mathfrak m_R$ is nilpotent, this implies
$(\mathfrak m_RA_R)_{s_i}=0$. Hence the homogeneous lift
$\widetilde x_i$ is unique.

The elements $N_s(\widetilde x,U_R)$ lift the monomial basis of
$A$ and form an $R$-basis by Lemma~\ref{lem:basis}. In particular,
the chosen elements generate $A_R$. Expanding the central relations
in this basis gives uniquely
\[
 g_\alpha(\widetilde x,U_R)=
 \sum_{j=1}^{\eta_{\Z,\alpha}}r_{\alpha,j}
 N_{d_\alpha-jm}(\widetilde x,U_R),
 \qquad r_{\alpha,j}\in R_{jm}.
\]
Consequently the lifted equations $G_\alpha$ of \eqref{eq:G}
vanish after setting $t_{\alpha,j}=r_{\alpha,j}$.
Recall that $R_{\alpha\beta}$ is the remainder obtained in
Section~\ref{sec:finite} by dividing the $S$-polynomial
$S_{\alpha\beta}$ by the ordered list $(G_1,\ldots,G_r)$.
The division identity expresses $S_{\alpha\beta}-R_{\alpha\beta}$
as a sum of multiples of the $G_\gamma$. Since the $S$-polynomial
and every $G_\gamma$ vanish in $A_R$, so does its remainder.
Since this remainder is a linear combination of standard monomials
whose images form an $R$-basis, each of its coefficients is zero.
Thus the assignment $t_{\alpha,j}\mapsto r_{\alpha,j}$ annihilates
$J_{\mathrm{coeff}}$ and these values
define a graded map $B_0\to R$ and an isomorphism of pair
deformations $\mathcal A_0\otimes_{B_0}R\xrightarrow{\sim}A_R$
sending $U$ to $U_R$.
Conversely, every such map gives a pair deformation by the
flatness in \eqref{eq:adaptedfree}. An isomorphism of pairs preserves
$U_R$ and all the other
uniquely determined generator lifts, hence all the coefficients
$r_{\alpha,j}$. The classifying map is unique.

The algebraic family already classifies pair deformations over
every allowed Artinian base $R$. We use the completion
$\widehat B_0$ to express this universal property over a complete
local parameter algebra, as required for prorepresentability.
Indeed, every graded map $B_0\to R$ sends $\mathfrak m_0$ into
$\mathfrak m_R$. If $\mathfrak m_R^N=0$, it factors through
$B_0/\mathfrak m_0^N$ and extends uniquely to a continuous local
graded map $\widehat B_0\to R$. Restriction to $B_0$ gives the
inverse correspondence, and pullback of $\widehat\pi_0$ gives
the same pair deformation. Under a graded base map $R\to R'$,
the coefficients $r_{\alpha,j}$ in the displayed expressions for
$g_\alpha(\widetilde x,U_R)$ map to their images in $R'$.
The standard monomials remain a basis after base change, so these
are precisely the coefficients determining the classifying map
of the pulled-back pair. Hence the correspondence is natural in
$R$, giving the bijection
$h_{\widehat B_0}(R)\xrightarrow{\sim}\Def_m^{(X,V(\chi^m))}(R)$.\qedhere
\end{proof}

\subsection{Universality for \texorpdfstring{$\Def_m^X$}{Def of X}}\label{subsec:minimal}

We first show that forgetting the divisor gives a formally versal
family of $X$, and then eliminate the redundancy in the choice
of divisor to obtain a formally universal family.

\begin{proposition}\label{prop:forgetful-smooth}
The forgetful morphism
\[
 \Def_m^{(X,V(\chi^m))}\longrightarrow\Def_m^X
\]
is formally smooth. Consequently, the family $\widehat\pi_0$
over $\widehat B_0$, with its divisor forgotten, is formally
versal for $\Def_m^X$, and its Kodaira--Spencer map is surjective
in every degree.
\end{proposition}
\begin{proof}
Let $R'\twoheadrightarrow R$ be a small graded extension,
let $(A_R,I_R)$ be a pair deformation, and let $A_{R'}$ be a
deformation of $X$ with a prescribed identification
$A_{R'}\otimes_{R'}R\simeq A_R$.
By Lemma~\ref{lem:divisor-lifts}, write $I_R=(U_R)$ with
$U_R$ a homogeneous lift of $\chi^m$. Choose a homogeneous lift
$U_{R'}\in(A_{R'})_m$ of $U_R$. The same lemma shows that
$A_{R'}/(U_{R'})$ is $R'$-flat and that the pair
$(A_{R'},(U_{R'}))$ restricts to $(A_R,I_R)$. This proves
formal smoothness.

Proposition~\ref{prop:marked-pair} identifies the source functor
with $h_{\widehat B_0}$, so formal versality follows.
Applying the lifting property to $D_j\twoheadrightarrow\C$
gives surjectivity of the Kodaira--Spencer map in degree $-jm$.
\end{proof}

After forgetting the divisor, any other homogeneous lift of $\chi^m$
has the form $U'_R=U_R+\lambda$, with $\lambda\in R_m$:
indeed, $(\mathfrak m_RA_R)_m=R_m$, since $(A_R)_0=\C$ and
$(1-j)m\notin S$ for $j\ge2$.
For the polynomials in \eqref{eq:finiteeta}, this gives
\[
 p'_\beta(U')=p_\beta(U'-\lambda),\qquad
 t'_{\beta,1}=t_{\beta,1}-\eta_{\Z,\beta}\lambda
 \quad(\eta_{\Z,\beta}>0).
\]
If $B_0=\C$, set $B=\C$ and $\mathcal A=A$.
For the following normalization, we may therefore assume $B_0\ne\C$.
Choose any $\beta_0$ with $\eta_{\Z,\beta_0}>0$.
The unique choice $\lambda=t_{\beta_0,1}/\eta_{\Z,\beta_0}$
makes $t'_{\beta_0,1}=0$, so we restrict to
\begin{equation}\label{eq:translation-condition}
 t_{\beta_0,1}=0.
\end{equation}

Over $B_0[\lambda]$, the binomial formula gives an invertible
change from the basis $(U^qN_b)$ to $((U+\lambda)^qN_b)$,
with inverse obtained by replacing $\lambda$ with $-\lambda$.
Thus the changed equations still have a standard-monomial basis,
so their $S$-polynomial remainders vanish. The new coefficients
$t'=(t'_{\beta,j})$ of the shifted polynomials $p'_\beta$ therefore
satisfy all equations defining $B_0$: substituting $t=t'$ into
any generator of $J_{\mathrm{coeff}}$ gives zero in $B_0[\lambda]$.
Each coefficient tuple is uniquely recovered from its tuple on
\eqref{eq:translation-condition} and
$t_0=t_{\beta_0,1}/\eta_{\Z,\beta_0}$: the inverse substitutes
$U+t_0$ in the restricted polynomials. Hence
\begin{equation}\label{eq:principalproduct}
 \begin{gathered}
 B_0\simeq B[t_0],\qquad B=B_0/(t_{\beta_0,1}),\qquad
 t_0=t_{\beta_0,1}/\eta_{\Z,\beta_0},\\
 \mathcal A=\mathcal A_0\otimes_{B_0}B.
 \end{gathered}
\end{equation}
A different choice of $\beta_0$ gives an isomorphic base and
family of $X$ by the same substitutions.
With $U'=U+t_0$, the universal pair family is the pullback of
$\mathcal A/B$ with divisor $V(U'-t_0)$.

Let $\mathfrak m_B$ be the homogeneous maximal ideal of $B$ and put
\[
 \widehat{\mathcal A}
 =\varprojlim_n\mathcal A/\mathfrak m_B^{n+1}\mathcal A,
 \qquad
 \widehat\pi:\operatorname{Spf}\widehat{\mathcal A}
 \longrightarrow\operatorname{Spf}\widehat B.
\]
Then $\widehat B_0\simeq\widehat B[[t_0]]$; after forgetting the
divisor, $\widehat\pi_0$ is the pullback of $\widehat\pi$.

\begin{corollary}\label{cor:universal-X}
The deformation family $\widehat\pi$ is formally universal for
$\Def_m^X$.
\end{corollary}
\begin{proof}
Every equivariant deformation of $X$ admits a
homogeneous lift $U_R$ of $\chi^m$, and
$(A_R,(U_R))$ is a deformation of the pair by
Lemma~\ref{lem:divisor-lifts}.
Proposition~\ref{prop:marked-pair} gives a classifying map from
$\widehat B_0$, and the preceding construction shows that its
underlying deformation is induced from $\widehat B$.
Suppose two such maps give isomorphic deformations.
The isomorphism fixes every other minimal generator lift and
sends $U$ to $U+\lambda$, with $\lambda\in R_m$.
The normalization above assigns the same normalized coefficient
tuple to tuples related by such a translation. Thus the two maps
agree on all coefficient parameters and are equal. This proves
$h_{\widehat B}\simeq\Def_m^X$.
\end{proof}

\begin{example}\label{ex:eighth-generator-change}
Return to Example~\ref{ex:eighth}. Choose $\beta_0=2$;
since $p_2(U)=U+t_{2,1}$, the required shift is
$t_0=t_{2,1}$ and $U'=U+t_0$.
Substituting $U=U'-t_0$ gives
\[
 z_1^2=(U'+t_{1,1}-t_0)z_0,\qquad z_3^2=U'z_4;
\]
the remaining equations use $p_\alpha(U'-t_0)$ in the same way.
Using their coefficients as the remaining parameters gives
\[
 B=B_0/(t_{2,1}),\qquad B_0\simeq B[t_{2,1}].
\]
\end{example}

\subsection{Comparison with hulls}\label{subsec:usual-miniversal}

We now work on $\operatorname{Art}$, the category of local Artinian
$\C$-algebras with residue field $\C$ and local $\C$-algebra
homomorphisms.
The functors $\Def^X:\operatorname{Art}\to\mathrm{Sets}$ and,
when $m\in S$, $\Def^{(X,V(\chi^m))}:\operatorname{Art}\to\mathrm{Sets}$
are defined by the data and conditions of Definition~\ref{def:pair}
after omitting all grading requirements (see, e.g., \cite[Definition~3.1]{FilipPair}
for pairs).
Their $R$-algebra isomorphisms respect the special-fibre identifications
and, for pairs, preserve the divisor ideal.

Formal versality and miniversality are defined as in
Subsection~\ref{subsec:pairs}, with $\operatorname{Art}_m$ replaced
by $\operatorname{Art}$ and all grading requirements omitted.

For $Y=X$ or $Y=(X,V(\chi^m))$, write $\Def^Y$ for the
corresponding functor on $\operatorname{Art}$, with no torus
actions prescribed. Write $\Def_m^Y:\operatorname{Art}_m\to\mathrm{Sets}$
for the corresponding functor $\Def_m^X$ or
$\Def_m^{(X,V(\chi^m))}$ of Definition~\ref{def:pair}.

For a complete local Noetherian $\C$-algebra $H$ with residue
field $\C$, let $h_H(R)$ be the set of continuous local
$\C$-algebra homomorphisms $H\to R$, with no grading condition.
A \emph{hull} of $\Def^Y$ is a formally smooth natural
transformation $h_H\to\Def^Y$ inducing an isomorphism on tangent
spaces, that is, a bijection on $\C[\epsilon]/(\epsilon^2)$
\cite[Tag~06G8]{Stacks}.
Equivalently, it is a miniversal formal deformation of $Y$ over
$H$. We call $H$ its \emph{hull ring}. A hull need not
prorepresent $\Def^Y$: classifying maps need not be unique.

Suppose $\Def^Y$ admits a hull. Since $\mathbb T$ is linearly
reductive, Rim's theorem \cite{Rim} allows us to choose its
miniversal formal family torus-equivariantly. Denote the hull
ring by $H_Y$. Homogeneous lifts of a weight basis of
$\mathfrak m_{H_Y}/\mathfrak m_{H_Y}^2$ give a presentation
\[
 H_Y=\C[[t_1,\ldots,t_N]]/J_Y,
 \qquad J_Y\subset(t_1,\ldots,t_N)^2,
\]
with homogeneous minimal parameters.
Write
\begin{equation}\label{eq:ordinary-ray-base}
 H_{Y,m}=H_Y/(t_i:\deg t_i\notin\Z_{>0}m)
\end{equation}
for the closed formal subspace obtained by setting all other
parameters to zero. Restricting to this subspace is the meaning
of miniversality in degrees $-jm$ used in
\cite[Remark~5.1]{FilipPair}.

\begin{corollary}\label{cor:ordinary-pair}
If $\Def^{(X,V(\chi^m))}$ admits a hull with ring
$H_{(X,V(\chi^m))}$, chosen torus-equivariantly as above, then
$\widehat B_0\simeq H_{(X,V(\chi^m)),m}$.
Under this isomorphism, $\widehat\pi_0$ is the restriction of
the corresponding miniversal formal family to the parameter
degrees $jm$, $j\ge1$.
\end{corollary}
\begin{proof}
Put $Y=(X,V(\chi^m))$. We first check equivariant lifting;
the same argument applies to $Y=X$.
Fix equivariant lifting data across a small graded extension
$R'\twoheadrightarrow R$ with kernel $E$.
Formal smoothness of the hull transformation
$h_{H_Y}\to\Def^Y$ supplies a lift of the prescribed base map
$H_Y\to R$. The lifts $H_Y\to R'$ form an affine space:
their differences are precisely the linear maps
$\mathfrak m_{H_Y}/\mathfrak m_{H_Y}^2\to E$.
Linear reductivity gives a fixed point for the induced
$\mathbb T$-action, hence an equivariant lift.
Keeping the identification over $R$ fixed, its difference from
the prescribed deformation lies in $(E\otimes_\C T_Y^1)^{\mathbb T}$.
The equivariant Kodaira--Spencer isomorphism supplies an invariant
correction cancelling this difference \cite{Rim}.
The comparison isomorphisms extending the fixed identification
form a torsor under $E\otimes_\C T_Y^0$, so linear reductivity
also supplies an equivariant isomorphism.
Here $T_Y^1$ and $T_Y^0$ denote the tangent and
infinitesimal automorphism spaces of the deformation problem for $Y$; for pairs, $T_Y^0$ consists
of derivations preserving the divisor ideal. All these modules
are rational $\mathbb T$-modules.

Every graded map $H_Y\to R$, with $R\in\operatorname{Art}_m$,
factors through $H_{Y,m}$. Thus the restricted family is formally
versal for $\Def_m^Y$. Since $J_Y\subset(t_1,\ldots,t_N)^2$,
its Kodaira--Spencer map is an isomorphism onto
$\bigoplus_{j\ge1}T_Y^1(-jm)$, so it is miniversal.

Set $C=H_{Y,m}$ and $D=\widehat B_0$.
Universality in Proposition~\ref{prop:marked-pair} gives a graded
classifying map $D\to C$. Equivariant versality gives compatible
classifying maps $C\to D/\mathfrak m_D^n$, hence a map $C\to D$.
Both induce tangent-space isomorphisms, so their composites are
automorphisms of complete Noetherian local rings.
Consequently the bases are isomorphic, and the classifying
isomorphisms identify their formal families.
\end{proof}

\begin{corollary}\label{cor:ordinary-X}
If $\Def^X$ admits a hull with ring $H_X$, chosen
torus-equivariantly as above, then $\widehat B\simeq H_{X,m}$,
and $\widehat\pi$ is the restriction of the corresponding
miniversal formal family to the parameter degrees $jm$, $j\ge1$.
\end{corollary}
\begin{proof}
Corollary~\ref{cor:universal-X} gives the lifting property and a
bijective Kodaira--Spencer map. The restriction of the chosen miniversal formal family over
$H_X$ has the same properties, by the argument of
Corollary~\ref{cor:ordinary-pair}. The comparison argument at the end of that proof,
with $D=\widehat B$, gives the assertion.
This comparison applies to any primitive $m$ once
$\widehat B$ prorepresents $\Def_m^X$, so together with
Theorem~\ref{thm:main} it proves Theorem~\ref{thm:usual-miniversality}.
\end{proof}

\begin{corollary}\label{cor:ordinary-isolated}
Suppose $X$ is an isolated Gorenstein toric singularity of
dimension at least three and $m=R^*$. Then $\widehat\pi$ is
formally miniversal for $\Def^X$.
\end{corollary}
\begin{proof}
The functor $\Def^X$ admits a hull and
$T_X^1=T_X^1(-R^*)$ \cite{Altmann}. Thus $H_{X,m}=H_X$, and
Corollary~\ref{cor:ordinary-X} applies.
\end{proof}

\section{Components and lattice-friendly Minkowski decompositions}\label{sec:components}

In this section, we describe the reduced components of the universal
equivariant bases for the pair and for $X$ by lattice-friendly
Minkowski decompositions. Throughout, we assume that $m\in S$ is primitive.

\subsection{Minkowski decompositions and multiplication}\label{subsec:pair-multiplication}

Choose a sublattice $M\subset\LL$ such that $\LL=M\oplus\Z m$,
and put $N=\operatorname{Hom}_{\Z}(M,\Z)$.
In the dual coordinates $\LL_{\mathbb R}^*=N_{\mathbb R}\oplus\mathbb R$,
evaluation by $m$ is the last coordinate. The projection
$(v,r)\mapsto v$ identifies each hyperplane $[m=r]$ with
$N_{\mathbb R}$. Under these identifications, we write the slice
and its tail cone as
\[
 P=P_m=\{v\in N_{\mathbb R}:(v,1)\in\sigma\},\qquad
 \kappa=\{v\in N_{\mathbb R}:(v,0)\in\sigma\}.
\]

\begin{definition}\label{def:latticefriendly}
A rational Minkowski decomposition
\[
 \mathcal D:\qquad P=Q_1+\cdots+Q_k,\qquad
 \operatorname{tail}(Q_i)=\kappa,
\]
is \emph{lattice-friendly} if, for every $c\in\kappa^\vee\cap M$, all but at most one of the minimizing faces
\[
 \operatorname{face}(Q_i,c)
 =\{v\in Q_i:\langle c,v\rangle=\min_{w\in Q_i}\langle c,w\rangle\}
\]
contain a point of $N$. This is \cite[Definition 10.6]{ACFpoly}. Equivalently, for each vertex decomposition $v=v_1+\cdots+v_k$ induced by a minimizing linear form, at most one $v_i$ is outside $N$: one obtains the vertex condition by choosing a form in a full-dimensional normal cone, and the face condition by refining a minimizing face to a vertex.

Decompositions are considered up to permutation and integral translations
of the summands whose sum is zero. A summand $n+\kappa$ with $n\in N$
is \emph{trivial}. We always remove such a summand when another summand
$Q_j$ is present, replacing $(n+\kappa)+Q_j$ by $n+Q_j$.
A refinement is \emph{proper} if, after removing trivial summands,
it replaces at least one summand by two or more nontrivial summands.
A decomposition is \emph{maximal} if it admits no proper
lattice-friendly refinement.
\end{definition}

\begin{remark*}
Ilten--Vollmert \cite{IltenVollmert}
constructed homogeneous deformations of toric varieties in arbitrary
degree from admissible Minkowski decompositions. Their vertex
condition \cite[Definition 2.2]{IltenVollmert} agrees with the
lattice-friendly condition used here.
\end{remark*}

We express the multiplication already present in $\mathcal A_0$
in terms of its $B_0[U]$-basis from \eqref{eq:adaptedfree}.
The coefficients of these products are polynomials in $U$,
computed by division by the equations $G_\alpha$.
In the next subsection, their factorizations will recover
the Minkowski decompositions that describe the reduced components.

For the fixed splitting and slice, put
\[
 C=\kappa^\vee\cap M,\qquad
 \eta(c)=-\min_{v\in P_m}\langle c,v\rangle,\qquad
 \eta_{\Z}(c)=\lceil\eta(c)\rceil.
\]
Here $\eta$ is a rational piecewise linear function, and
$\eta_{\Z}$ is its integral rounding, as in \cite{ACF}.
For $c\in C$, the inequalities defining $\sigma^\vee$ say that
$(c,r)\in S$ exactly when $r\in\Z$ and $r\ge\eta(c)$. Thus
\begin{equation}\label{eq:etaS}
 S=\{(c,r):c\in C,\ r\ge\eta_{\Z}(c)\}.
\end{equation}
In particular, $(c,\eta_{\Z}(c))$ is the unique exponent of
$\partial_{\N m}S$ projecting to $c$. Write
$W_c=N_{(c,\eta_{\Z}(c))}$ for its standard monomial representative
from Section~\ref{sec:finite}. Thus the monomials $W_c$, which
contain no $U$, index the boundary basis by $C$.
Also, $C-C=M$, since $\kappa$ contains no line.
For $c,d\in C$, put
\[
 \eta_{\Z}(c,d)
 :=\eta_{\Z}(c)+\eta_{\Z}(d)-\eta_{\Z}(c+d)\ge0.
\]
With $u=\chi^m$ and $e_c=\chi^{(c,\eta_{\Z}(c))}$, addition
of exponents gives the central-fibre multiplication
\begin{equation}\label{eq:etamult}
 A=\bigoplus_{c\in C}\C[u]e_c,\qquad
 e_ce_d=u^{\eta_{\Z}(c,d)}e_{c+d},\qquad e_0=1.
\end{equation}

In $\mathcal A_0$, use $e_c$ also for the residue class of $W_c$.
By \eqref{eq:adaptedfree},
$\mathcal A_0=\bigoplus_{c\in C}B_0[U]e_c$.
Since $U$ and all base parameters have degrees in $\Z m$,
the product $e_ce_d$ is a unique polynomial multiple of $e_{c+d}$.
Homogeneity determines its possible powers of $U$, and
\eqref{eq:etamult} fixes its leading coefficient to $1$, since
$(B_0)_0=\C$. Thus
\begin{equation}\label{eq:F}
 e_ce_d=p_{c,d}(U)e_{c+d},\qquad
 p_{c,d}(U)=U^{\eta_{\Z}(c,d)}+\sum_{j=1}^{\eta_{\Z}(c,d)}t_{c,d,j}U^{\eta_{\Z}(c,d)-j},
 \quad t_{c,d,j}\in(B_0)_{jm}.
\end{equation}

\subsection{Polynomial factorization and components}\label{subsec:factorization-components}

We now apply this description of multiplication to the
Gr\"obner equations \eqref{eq:finiteeta}.
For a lattice-friendly decomposition $\mathcal D:P=Q_1+\cdots+Q_k$, put
\[
 \eta_{i,\Z}(c)=\left\lceil-\min_{v\in Q_i}\langle c,v\rangle\right\rceil.
\]
These functions are subadditive. Lattice-friendliness gives
$\eta_{\Z}(c)=\sum_i\eta_{i,\Z}(c)$, since all but at most
one of the minimizing values are integral.

For the $\beta$th Gr\"obner equation,
write $\LM(g_\beta)=x^{\boldsymbol\alpha_\beta}
=x_1^{\alpha_{\beta,1}}\cdots x_n^{\alpha_{\beta,n}}$,
where $\alpha_{\beta,j}$ is the exponent of $x_j$.
Put $\deg x_j=(c_j,\eta_{\Z}(c_j))$ and
$c_\beta=\sum_j\alpha_{\beta,j}c_j$, so $W_\beta=W_{c_\beta}$.
The exponent of $U$ in the central equation splits as
\[
\begin{aligned}
 n_{i\beta}&=\sum_j\alpha_{\beta,j}\eta_{i,\Z}(c_j)
                    -\eta_{i,\Z}(c_\beta)\ge0,\\
 \eta_{\Z,\beta}
 &=\sum_j\alpha_{\beta,j}\eta_{\Z}(c_j)-\eta_{\Z}(c_\beta)
 =\sum_i n_{i\beta}.
\end{aligned}
\]
Introduce parameters $\lambda_1,\ldots,\lambda_k$ of degree $m$
and replace $U^{\eta_{\Z,\beta}}$ by
$\prod_i(U-\lambda_i)^{n_{i\beta}}$.
Comparing coefficients with $p_\beta(U)$ in \eqref{eq:finiteeta}
defines the images of the original parameters:
\[
 U^{\eta_{\Z,\beta}}
 +\sum_{j=1}^{\eta_{\Z,\beta}}
       \psi_{\mathcal D}^*(t_{\beta,j})U^{\eta_{\Z,\beta}-j}
 =\prod_i(U-\lambda_i)^{n_{i\beta}}.
\]
It remains to check that these coefficients satisfy the equations
of $B_0$.

\begin{proposition}\label{prop:decomposition-family}
The assignment defines a graded homomorphism
$\psi_{\mathcal D}^*:B_0\to\C[\lambda_1,\ldots,\lambda_k]$.
The pullback of the pair family along
$\psi_{\mathcal D}:\mathbb A^k\to\Spec B_0$ is flat, with multiplication
\begin{equation}\label{eq:decomposition-family}
 e_ce_d=\prod_{i=1}^k(U-\lambda_i)^%
 {\eta_{i,\Z}(c)+\eta_{i,\Z}(d)-\eta_{i,\Z}(c+d)}e_{c+d}.
\end{equation}
\end{proposition}
\begin{proof}
Put $\Lambda=\C[\lambda_1,\ldots,\lambda_k]$. Substitute for
$t_{\beta,j}$ the coefficient of $U^{\eta_{\Z,\beta}-j}$ in
$\prod_i(U-\lambda_i)^{n_{i\beta}}$. We verify Buchberger's
criterion for the resulting equations
\[
 \widetilde G_\beta
 =x^{\boldsymbol\alpha_\beta}
  -\prod_i(U-\lambda_i)^{n_{i\beta}}W_\beta.
\]
Apply the division procedure of Section~\ref{sec:finite} to a monomial
$x^{\boldsymbol a}$, treating $U$ as part of the coefficient ring
$\Lambda[U]$. At each step, replace a factor
$x^{\boldsymbol\alpha_\beta}$ by
$\prod_i(U-\lambda_i)^{n_{i\beta}}W_\beta$.
The monomial in the $x$-variables strictly decreases in the order
induced by Section~\ref{sec:finite}, so the process terminates at the
unique standard monomial $W_d$, where $d=\sum_j a_jc_j$.
Write $W_d=x^{\boldsymbol b}$. It represents
$\chi^{(d,\eta_{\Z}(d))}$, so equality of degrees gives
$d=\sum_j b_jc_j$ and
$\eta_{\Z}(d)=\sum_j b_j\eta_{\Z}(c_j)$.
For each $i$, subadditivity gives
$\eta_{i,\Z}(d)\le\sum_j b_j\eta_{i,\Z}(c_j)$.
The sum of these inequalities is an equality, since
$\eta_{\Z}=\sum_i\eta_{i,\Z}$. Hence each is an equality:
\[
 \sum_j b_j\eta_{i,\Z}(c_j)=\eta_{i,\Z}(d)
 \qquad\text{for every }i.
\]
The same identity for $W_\beta$ shows that a division step using
$\widetilde G_\beta$ decreases the sum of the $\eta_{i,\Z}(c_j)$
over the factors of the current monomial by $n_{i\beta}$,
while multiplying its coefficient by $(U-\lambda_i)^{n_{i\beta}}$.
Thus the remainder of $x^{\boldsymbol a}$ is
\[
 \prod_i(U-\lambda_i)^{
       \sum_j a_j\eta_{i,\Z}(c_j)-\eta_{i,\Z}(d)}W_d,
\]
independently of the choices made during division.
Each $S$-polynomial compares the two first division steps
obtained by cancelling the same least common multiple with
two different equations. The resulting expressions have the same
remainder by the displayed formula, so every $S$-polynomial has
zero remainder.
Buchberger's criterion therefore shows that the
$\widetilde G_\beta$ form a monic Gr\"obner basis over
$\Lambda[U]$, with standard monomials $W_d$.

To check the equations of $B_0$, recall that each remainder
$R_{\alpha\beta}$ used to define $J_{\mathrm{coeff}}$ in
Section~\ref{sec:finite} is obtained from an $S$-polynomial
by subtracting multiples of the $G_\beta$.
It therefore belongs to $(G_1,\ldots,G_r)$ and is a linear
combination of standard monomials.
After replacing the parameters by their prescribed polynomials
in the $\lambda_i$, it still has only standard monomials and
belongs to $(\widetilde G_1,\ldots,\widetilde G_r)$.
The Gr\"obner-basis property forces it to be zero.
Since $U$ remains an indeterminate, every coefficient defining
$J_{\mathrm{coeff}}$ vanishes,
and the assignment factors through $B_0$; it is graded by construction.
The pullback and its quotient by $(U)$ are free by
\eqref{eq:adaptedfree}, proving flatness of the pair.
Applying this division procedure to $W_cW_d$ gives
\eqref{eq:decomposition-family}.
\end{proof}
For the homomorphism $\psi_{\mathcal D}^*$ of
Proposition~\ref{prop:decomposition-family}, set
\begin{equation}\label{eq:pair-component-family}
 B_{0,\mathcal D}
 =\operatorname{im}(\psi_{\mathcal D}^*),
 \qquad
 \mathcal A_{0,\mathcal D}
 =\mathcal A_0\otimes_{B_0}B_{0,\mathcal D}.
\end{equation}
This family is free by base change. Complete $B_{0,\mathcal D}$
at its ideal of positive-degree elements, and $\mathcal A_{0,\mathcal D}$
at the ideal generated by these base parameters.
The resulting formal pair family, with marked divisor defined by $U=0$,
is the pullback of $\widehat\pi_0$ to
$\operatorname{Spf}\widehat B_{0,\mathcal D}$.

For $0\ne f\in\C(U)$, the order
$\operatorname{ord}_a f$ is the unique integer $r$ such that
$f(U)=(U-a)^r h(U)$ with $h$ regular and nonzero at $a$.
The next proposition recovers a decomposition by taking these orders
in the multiplication polynomials.

\begin{proposition}\label{prop:recover}
Assume that $\eta_{\Z,\beta}>0$ for some $\beta$.
Then every closed point of $\Spec B_0$ is obtained from a
lattice-friendly decomposition into nontrivial summands by
\eqref{eq:decomposition-family}.
If the parameter values attached to its summands are pairwise
distinct, their multiplicity functions determine the decomposition
up to permutation and integral translations of the summands whose sum is zero.
\end{proposition}
\begin{proof}
Fix a closed point $q\in\Spec B_0$. This amounts to choosing complex
values $t_{\beta,j}(q)$ satisfying the equations defining $B_0$.
Substitute $t_{\beta,j}(q)$ for each parameter $t_{\beta,j}$ in the
family equations. In particular,
\[
 p_\beta(U)=U^{\eta_{\Z,\beta}}
 +\sum_{j=1}^{\eta_{\Z,\beta}}
 t_{\beta,j}(q)U^{\eta_{\Z,\beta}-j}\in\C[U],
\]
and the fibre over $q$ has multiplication
\[
 e_ce_d=p_{c,d}(U)e_{c+d}.
\]
Recall that $\beta$ indexes the defining equations from
Section~\ref{sec:finite}, whereas $c,d\in C$ index basis elements
$e_c,e_d$.
Throughout this proof, $p_\beta$ and $p_{c,d}$ denote the polynomials
after this substitution. Their coefficients are now complex numbers, while $U$ remains
an indeterminate. Their leading coefficient is still $1$, so
their degrees are unchanged.
We will recover one Minkowski summand for each distinct root
of these polynomials and use that root as its parameter value.

First, we prove that
\begin{equation}\label{eq:root-sets}
 \bigcup_{c,d\in C}\{a\in\C:p_{c,d}(a)=0\}
 =\bigcup_{\beta=1}^{r}\{a\in\C:p_\beta(a)=0\}.
\end{equation}
To compute $p_{c,d}$, apply the division procedure of
Section~\ref{sec:finite} to $W_cW_d$, treating $U$ as part of the
coefficient ring $\C[U]$. By \eqref{eq:finiteeta}, each step replaces
a factor $\LM(g_\beta)$ by $p_\beta(U)W_\beta$. The monomial in the
other variables strictly decreases, and the remainder is
$p_{c,d}(U)W_{c+d}$. Since each step contributes a factor $p_\beta(U)$,
$p_{c,d}$ is a product of the $p_\beta$.

Conversely, fix a defining equation $G_\beta=0$. Since
$\LM(g_\beta)=\prod_jx_j^{\alpha_{\beta,j}}$, $x_j=e_{c_j}$ and
$W_\beta=e_{c_\beta}$ in the fibre algebra, this equation reads
\[
 \prod_j e_{c_j}^{\alpha_{\beta,j}}=p_\beta(U)e_{c_\beta}.
\]
Compute the same product by multiplying its factors one at a time
using \eqref{eq:F}. Each multiplication contributes a polynomial
$p_{c,d}(U)$, and the final basis element is $e_{c_\beta}$ because
$c_\beta=\sum_j\alpha_{\beta,j}c_j$. As the $e_c$ form a basis over
$\C[U]$, the coefficient obtained in this way must equal $p_\beta(U)$.
Thus every $p_\beta$ is a product of the $p_{c,d}$, and the
multiplication polynomials recover all original parameters
$t_{\beta,j}(q)$ as the coefficients of the $p_\beta$.
The two product statements prove \eqref{eq:root-sets}.
This common set is finite, since there are finitely many $p_\beta$,
and nonempty because some $p_\beta$ has positive degree.
Denote the distinct roots by $a_1,\ldots,a_k$ and write
\[
 \eta_{i,\Z}(c,d)=\operatorname{ord}_{a_i}p_{c,d},\qquad
 p_{c,d}(U)=\prod_i(U-a_i)^{\eta_{i,\Z}(c,d)}.
\]
These multiplicities are nonnegative integers whose sum is
$\deg p_{c,d}=\eta_{\Z}(c,d)$.

Next, we seek integer-valued functions $g_i:C\to\Z$ such that
\begin{equation}\label{eq:multiplicity-difference}
 \eta_{i,\Z}(c,d)=g_i(c)+g_i(d)-g_i(c+d)
 \qquad(c,d\in C).
\end{equation}
To construct them, put $K=\C(U)$, where every $p_{c,d}$ is invertible.
Multiplication by $e_c$ is injective: it sends each basis element
$e_d$ to a nonzero multiple of $e_{c+d}$. We may therefore invert
all $e_c$. Since $C-C=M$, every degree of the resulting algebra
contains a unit $e_c/e_d$. The multiplication rule shows that
any two such fractions of the same degree are proportional over $K$,
so each homogeneous component is one-dimensional. Choose units
in the degrees of a lattice basis of $M$. Their products and
inverses give Laurent monomials $x^c$ for $c\in M$, satisfying
$x^cx^d=x^{c+d}$. Each $e_c$ is a nonzero scalar
multiple of $x^c$, and comparison with \eqref{eq:F} gives
\[
 e_c=f_c(U)x^c,\qquad
 p_{c,d}(U)=\frac{f_c(U)f_d(U)}{f_{c+d}(U)},
 \qquad f_c(U)\in K^\times.
\]
Set $g_i(c)=\operatorname{ord}_{a_i}f_c$. Taking orders at $a_i$
in the quotient formula proves \eqref{eq:multiplicity-difference},
since $\operatorname{ord}_{a_i}p_{c,d}=\eta_{i,\Z}(c,d)$.
In particular, the $g_i$ are subadditive because these multiplicities
are nonnegative. Put $L(c)=\sum_i g_i(c)-\eta_{\Z}(c)$.
Since $\sum_i\eta_{i,\Z}(c,d)=\eta_{\Z}(c,d)$, summing
\eqref{eq:multiplicity-difference} over $i$ gives
$L(c+d)=L(c)+L(d)$. As $C-C=M$, the rule
$L(c-d)=L(c)-L(d)$ extends $L$ uniquely to an integral linear form
on $M$. Replace $g_1$ by $g_1-L$, leaving the other $g_i$ unchanged
and retaining the same notation. Additivity of $L$ ensures that
\eqref{eq:multiplicity-difference} is unchanged, while now
$\sum_i g_i=\eta_{\Z}$.

We now turn the $g_i$ into support functions of polyhedra.
Choose an integer $D\ge2$ such that $Dv\in N$ for every vertex
$v$ of $P$. This removes rounding on $DC$, since
$\eta_{\Z}(Dc)=\eta(Dc)=D\eta(c)$. Put
\[
 \eta_i(c)=D^{-1}g_i(Dc)\qquad(c\in C),
 \qquad\sum_i\eta_i(c)=\eta(c).
\]
On a full-dimensional normal cone $\nu$ of $P$, the function
$\eta$ is linear. Hence, for $c,d\in\nu\cap M$,
\[
 \sum_i\bigl(g_i(Dc)+g_i(Dd)-g_i(D(c+d))\bigr)=0.
\]
Every summand is nonnegative by subadditivity, so each is zero.
Consequently $g_i(D\,\cdot)$ is additive
on $\nu\cap M$. This monoid generates $M$ as a group:
for $w\in\operatorname{int}(\nu)\cap M$ and $a\in M$,
write $a=(a+nw)-nw$ with both terms in $\nu\cap M$ for $n\gg0$.
Thus $D^{-1}g_i(D\,\cdot)$ extends uniquely to a linear form
$\ell_{i,\nu}$ taking values in $D^{-1}\Z$ on $M$.

On intersections of normal cones, the linear forms agree on
lattice points, since they extend the same function
$D^{-1}g_i(D\,\cdot)$. These lattice points span the intersection,
so the forms agree everywhere on it. They therefore
glue to a continuous piecewise linear function $\eta_i$ on
$\kappa^\vee$. Subadditivity and $\sum_i\eta_i=\eta$ hold
on lattice points; scaling and continuity extend them first
to rational points and then to all of $\kappa^\vee$.

Write $\ell_{i,\nu}(c)=-\langle c,v_{i,\nu}\rangle$ with
$v_{i,\nu}\in D^{-1}N$. Each linear piece is bounded above
by $\eta_i$ everywhere: for $y\in\operatorname{int}(\nu)$,
$c\in\kappa^\vee$ and sufficiently small $\varepsilon>0$,
we have $y-\varepsilon c\in\nu$, and
\[
 \eta_i(y)\le\eta_i(y-\varepsilon c)+\varepsilon\eta_i(c)
 =\eta_i(y)-\varepsilon\ell_{i,\nu}(c)+\varepsilon\eta_i(c).
\]
Hence $\eta_i=\max_\nu\ell_{i,\nu}$. Define
\[
 Q_i=\{v\in N_{\mathbb R}:\langle c,v\rangle\ge-\eta_i(c)
                         \text{ for all }c\in\kappa^\vee\}.
\]
Then
\[
 Q_i=\operatorname{conv}\{v_{i,\nu}:\nu\text{ a full-dimensional
 normal cone of }P\}+\kappa.
\]
Thus $Q_i$ has tail $\kappa$, its vertices lie in $D^{-1}N$, and
$\eta_i(c)=-\min_{v\in Q_i}\langle c,v\rangle$.
Since support functions add under Minkowski sums,
$\sum_i\eta_i=\eta$ gives $P=\sum_iQ_i$.

We first show that
\begin{equation}\label{eq:rounding-recovery}
 \lceil\eta_i(c)\rceil=g_i(c)
 \qquad(c\in C,\ 1\le i\le k).
\end{equation}
Subadditivity gives $g_i(Dc)\le Dg_i(c)$, hence
$g_i(c)\ge D^{-1}g_i(Dc)=\eta_i(c)$. Together with
$\sum_i g_i=\eta_{\Z}$ and $\sum_i\eta_i=\eta$, this yields
\[
 0\le g_i(c)-\eta_i(c),\qquad
 \sum_i\bigl(g_i(c)-\eta_i(c)\bigr)
       =\lceil\eta(c)\rceil-\eta(c)<1.
\]
Since $g_i(c)$ is integral, these inequalities prove
\eqref{eq:rounding-recovery}; in particular, $g_i=\eta_{i,\Z}$.

We now show that the decomposition is lattice-friendly.
Put $\epsilon(c)=\lceil\eta(c)\rceil-\eta(c)\in[0,1)$.
We first prove that
\begin{equation}\label{eq:nonintegral-count}
 \#\{i:\eta_i(c)\notin\Z\}
 =\epsilon(c)+\epsilon((D-1)c)<2
 \qquad(c\in C).
\end{equation}
By \eqref{eq:rounding-recovery} and $\sum_i g_i=\eta_{\Z}$,
$\epsilon(c)$ is the sum of the rounding errors
$\lceil\eta_i(c)\rceil-\eta_i(c)$.
Since $D\eta_i(c)$ is integral and
$\eta_i((D-1)c)=(D-1)\eta_i(c)$, the two errors at $c$ and
$(D-1)c$ sum to $\lceil\eta_i(c)\rceil-\lfloor\eta_i(c)\rfloor$.
This is $0$ when $\eta_i(c)$ is integral and $1$ otherwise.
Summing over $i$ proves the equality in \eqref{eq:nonintegral-count};
the strict inequality follows because both values of $\epsilon$
are less than $1$. Thus at most one $\eta_i(c)$ is nonintegral
for each $c\in C$.

We verify the equivalent vertex condition of
Definition~\ref{def:latticefriendly}. Fix a full-dimensional normal
cone $\nu$ of $P$, and let $v_{i,\nu}$ be the corresponding
vertices of the $Q_i$. Suppose that $v_{i,\nu}$ and $v_{j,\nu}$,
with $i\ne j$, are both nonintegral. We will find a single
$c\in\operatorname{int}(\nu)\cap M$ for which both support values
$\eta_i(c)$ and $\eta_j(c)$ are nonintegral.
Since $Dv_{i,\nu},Dv_{j,\nu}\in N$, the pairings define homomorphisms
\[
 M/DM\longrightarrow\Q/\Z,\qquad
 [c]\longmapsto\langle c,v_{i,\nu}\rangle\bmod\Z
\]
and likewise for $j$. Nonintegrality of the vertices makes both
maps nonzero, so their kernels are proper subgroups.
Two proper subgroups cannot cover a group: if one contains the
other, their union is proper; otherwise, the sum of an element
from each set difference lies outside both.
Choose a residue class outside both kernels and a representative
$c_0\in M$. For $w\in\operatorname{int}(\nu)\cap M$ and $n\gg0$,
$c=c_0+nDw$ lies in $\operatorname{int}(\nu)$ and represents the
same class. Hence both
$\eta_i(c)=-\langle c,v_{i,\nu}\rangle$ and
$\eta_j(c)=-\langle c,v_{j,\nu}\rangle$ are nonintegral,
contradicting \eqref{eq:nonintegral-count}.
Thus every induced vertex decomposition has at most one
nonintegral summand, so the decomposition is lattice-friendly
by Definition~\ref{def:latticefriendly}.
Moreover, \eqref{eq:rounding-recovery} shows that
\eqref{eq:decomposition-family}, with $\lambda_i=a_i$, has exactly
the multiplicities used in the factorization of $p_{c,d}$ above.
It therefore recovers all multiplication polynomials. Hence the
parameters of the family associated with
$\mathcal D:P=Q_1+\cdots+Q_k$, evaluated at $\lambda_i=a_i$,
equal $t_{\beta,j}(q)$. Therefore
$\psi_{\mathcal D}(a_1,\ldots,a_k)=q$, as required.

Finally, we prove uniqueness and exclude trivial summands.
For pairwise distinct parameter values, polynomial factorization
determines the multiplicity functions root by root, up to permutation.
If $g_i$ and $g'_i$ give the same multiplicities, their
difference is additive and extends to an integral linear form
on $M=C-C$, so their summands differ by integral translations.
These translations sum to zero because the support functions
sum to $\eta$. A zero multiplicity function makes $g_i$ integral
linear, hence $\eta_i=g_i$ and $Q_i=n_i+\kappa$ for some $n_i\in N$.
Conversely, such a trivial summand has zero multiplicities.
Each chosen root $a_i$ occurs in some $p_{c,d}$, so its multiplicity
function is nonzero. Thus all recovered summands are nontrivial.
\end{proof}

\begin{proposition}\label{prop:componentimages}
For a lattice-friendly decomposition $\mathcal D$ into $k$ nontrivial
summands, the image $\mathcal C^0_{\mathcal D}$ of $\psi_{\mathcal D}$
is closed and irreducible of dimension $k$.
For any lattice-friendly decomposition
$\mathcal E$, we have $\mathcal C^0_{\mathcal D}\subseteq
\mathcal C^0_{\mathcal E}$ if and only if $\mathcal E$ refines
$\mathcal D$. Consequently, maximal lattice-friendly decompositions
of $P$ correspond bijectively to the irreducible components of
$(\Spec B_0)_{\mathrm{red}}$. Completion at the origin gives the
same correspondence for $(\Spec\widehat B_0)_{\mathrm{red}}$.
\end{proposition}
\begin{proof}
Assume first that $B_0\ne\C$. We first prove that each image
$\mathcal C^0_{\mathcal D}$ is closed and irreducible of the stated
dimension. Fix $\mathcal D:P=Q_1+\cdots+Q_k$ and recall that
$B_{0,\mathcal D}=\operatorname{im}\psi_{\mathcal D}^*$ is a
subring of $\Lambda=\C[\lambda_1,\ldots,\lambda_k]$.
For each nontrivial $Q_i$, some $n_{i\beta}$ is positive, by the
last argument of Proposition~\ref{prop:recover}. The polynomial
\[
 p_\beta(U)=\prod_{j=1}^k(U-\lambda_j)^{n_{j\beta}}
 \in B_{0,\mathcal D}[U]
\]
is monic and vanishes at $U=\lambda_i$. Thus every $\lambda_i$
is integral over $B_{0,\mathcal D}$, so $\Lambda$ is finite over
this ring. Consequently $\psi_{\mathcal D}$ is finite and
surjective onto $\Spec B_{0,\mathcal D}$, a closed subscheme of
$\Spec B_0$. Since $B_{0,\mathcal D}$ is a domain and its integral
extension $\Lambda$ has dimension $k$, this proves the assertion
about $\mathcal C^0_{\mathcal D}$.

Next, we show that these images form a finite cover of the reduced
base. The fixed division procedure expresses each $p_{c,d}$ as a
product of the $p_\beta$. Taking multiplicities therefore shows
that the vector $(n_{i\beta})_\beta$ determines the entire
multiplicity function of $Q_i$, hence $Q_i$ up to integral
translation by Proposition~\ref{prop:recover}. These vectors are
nonzero, have nonnegative integer entries, and satisfy
\[
 \sum_i n_{i\beta}=\eta_{\Z,\beta},\qquad
 k\le\sum_\beta\eta_{\Z,\beta}.
\]
Only finitely many collections of such vectors are possible.
Their decompositions are determined up to the equivalence of
Definition~\ref{def:latticefriendly}, since the translations sum
to zero when both decompositions sum to $P$.
Proposition~\ref{prop:recover} places every closed point of
$\Spec B_0$ in one of these images. Their union is closed, so
these finitely many closed irreducible images cover the reduced base.

We now prove the containment criterion
\begin{equation}\label{eq:component-refinement}
 \mathcal C^0_{\mathcal D}\subseteq\mathcal C^0_{\mathcal E}
 \quad\Longleftrightarrow\quad
 \mathcal E\text{ refines }\mathcal D.
\end{equation}
Write $\mathcal E:P=R_1+\cdots+R_s$, with parameters
$\mu_1,\ldots,\mu_s$. If $\mathcal E$ refines $\mathcal D$,
then, after the allowed integral translations, there is a
partition into nonempty blocks $J_1,\ldots,J_k$ such that
\[
 Q_i=\sum_{j\in J_i}R_j,\qquad
 n^{\mathcal D}_{i\beta}
 =\sum_{j\in J_i}n^{\mathcal E}_{j\beta}.
\]
The second equality follows because lattice-friendliness makes
the rounded support functions additive within each block.
Setting $\mu_j=\lambda_i$ for $j\in J_i$ in the family for
$\mathcal E$ therefore gives exactly the polynomials $p_\beta$
for $\mathcal D$, proving the required containment.

Conversely, assume the containment in
\eqref{eq:component-refinement}. Choose pairwise distinct
$\lambda_1,\ldots,\lambda_k$ and write
\[
 q=\psi_{\mathcal D}(\lambda_1,\ldots,\lambda_k)
  =\psi_{\mathcal E}(\mu_1,\ldots,\mu_s).
\]
Every parameter value on either side occurs as a root of some
$p_\beta$, since every summand is nontrivial. Equality of these
polynomials therefore shows that their distinct root values
coincide. Thus $J_i=\{j:\mu_j=\lambda_i\}$ form a partition into
nonempty blocks. Comparing multiplicities at $\lambda_i$ gives
$n^{\mathcal D}_{i\beta}=\sum_{j\in J_i}n^{\mathcal E}_{j\beta}$.
Grouping the $R_j$ by these blocks preserves lattice-friendliness,
and Proposition~\ref{prop:recover} identifies each grouped sum with
$Q_i$ up to integral translation. The translations sum to zero
because both decompositions sum to $P$. Hence $\mathcal E$ refines
$\mathcal D$, proving \eqref{eq:component-refinement}.
A proper refinement increases the number of nontrivial summands
and therefore the dimension of the image. Together with the
finite cover above, \eqref{eq:component-refinement} shows that
maximal decompositions correspond bijectively to the irreducible
components of the reduced base.

We next show that completion preserves these components.
For each maximal decomposition $\mathcal D$, put
$I_{\mathcal D}=\ker\psi_{\mathcal D}^*$, so
$B_0/I_{\mathcal D}=B_{0,\mathcal D}$.
Exactness of completion \cite[Tag 00MB]{Stacks} gives
\begin{equation}\label{eq:completed-component}
 \widehat B_0/I_{\mathcal D}\widehat B_0
 \simeq\widehat B_{0,\mathcal D}.
\end{equation}
The degree and maximal-ideal filtrations of the finitely generated
positively graded domain $B_{0,\mathcal D}$ are cofinal.
Its completion consists of formal sums of homogeneous elements,
and the product of two such nonzero sums has a nonzero lowest-degree
term. Hence $\widehat B_{0,\mathcal D}$ is a domain, and
$B_{0,\mathcal D}\to\widehat B_{0,\mathcal D}$ is injective.
By \eqref{eq:completed-component}, the extended ideals
$I_{\mathcal D}\widehat B_0$ are therefore prime and contract to
$I_{\mathcal D}$; in particular, they remain distinct and none
contains another. Their intersection is nilpotent: the intersection
of the $I_{\mathcal D}$ is the nilradical of the Noetherian ring
$B_0$, and exactness of completion preserves this finite intersection.
Thus the extended ideals are exactly the minimal primes of
$\widehat B_0$, proving the component correspondence after completion.
Their dimensions are unchanged, and completing
\eqref{eq:pair-component-family} gives their families.

Finally, suppose $B_0=\C$. All positive-degree parameters then
vanish. If some $\eta_{\Z,\beta}>0$, the one-summand family of
Proposition~\ref{prop:decomposition-family} would send
$t_{\beta,1}$ to the nonzero polynomial
$-\eta_{\Z,\beta}\lambda$, a contradiction.
Thus all $\eta_{\Z,\beta}=0$, and division gives $p_{c,d}=1$
for every $c,d$. In any lattice-friendly decomposition, the
nonnegative multiplicities of the summands therefore sum to zero,
so each is zero. By the last argument of
Proposition~\ref{prop:recover}, every summand is trivial.
Only the trivial decomposition remains, and both bases are single points.
\end{proof}

\begin{proof}[Proof of Theorem~\ref{thm:prorepresentability}]
Section~\ref{sec:finite} constructs the finitely generated parameter
algebra $B_0$ and the explicit family, which is free over $B_0[U]$
by Proposition~\ref{prop:free}. Proposition~\ref{prop:marked-pair}
proves formal universality. Proposition~\ref{prop:componentimages}
identifies the reduced formal components with maximal lattice-friendly
decompositions and gives dimension $k$ for a decomposition into $k$
nontrivial summands. Their families are the completions of the
explicit free families \eqref{eq:pair-component-family}.
\end{proof}

\begin{corollary}\label{cor:coefficient-components}
The irreducible components of $(\Spec\widehat B)_{\mathrm{red}}$
correspond bijectively to maximal lattice-friendly decompositions
of $P_m$. For a maximal decomposition $\mathcal D$ into $k>0$
nontrivial summands, the corresponding component has dimension
$k-1$ and is defined by $\widehat B_{\mathcal D}$, where
\[
 B_{\mathcal D}=B_{0,\mathcal D}/(t_0).
\]
Here $t_0$ is the degree-$m$ parameter in the product
$B_0\simeq B[t_0]$ of \eqref{eq:principalproduct}, recording
the choice of divisor.
The restriction of $\widehat\pi$ to this component is obtained by completing
$\mathcal A_{\mathcal D}
=\mathcal A_{0,\mathcal D}\otimes_{B_{0,\mathcal D}}B_{\mathcal D}$.
\end{corollary}
\begin{proof}
For $B_0\ne\C$, \eqref{eq:principalproduct} gives $B_0\simeq B[t_0]$.
Its minimal primes are extensions of those of $B$, giving
$B_{0,\mathcal D}\simeq B_{\mathcal D}[t_0]$.
Completion gives the component correspondence and the dimension
difference, while base change gives the stated family.
If $B_0=\C$, also $B=\C$: both bases are points, and the
only maximal decomposition is trivial, as explained in
Proposition~\ref{prop:componentimages}.
\end{proof}

\begin{remark}\label{rem:gorenstein-outlook}
If $X$ is Gorenstein and $m=R^*$, then
$V(\chi^m)=\partial X$, so the construction gives a deformation
of the pair $(X,\partial X)$ studied in \cite{FilipPair}.
In \cite{FilipLaurent}, we constructed deformation families of
three-dimensional Gorenstein toric pairs from Laurent polynomials
and their mutations. The families associated with $0$-mutable
polynomials are conjecturally families over the smoothing
components predicted by Corti--Filip--Petracci \cite{CFP}.
We hope that the results of the present paper will help extend
these constructions and their component interpretation to
non-Gorenstein toric varieties.
\end{remark}

\begin{example}\label{ex:eighth-equations}\label{ex:eighth-components}
Return to Example~\ref{ex:eighth}, retaining its parameters
$t_{\alpha,j}$. Eliminating the coefficients of $p_4,p_5,p_6$
from \eqref{eq:eighth-overlaps} gives
\[
 B_0\simeq
 \frac{\C[t_{1,1},t_{2,1},t_{3,1},t_{3,2}]}
 {\bigl((t_{1,1}-t_{2,1})(t_{2,1}^2-t_{3,1}t_{2,1}+t_{3,2})\bigr)}.
\]
Indeed, $p_5$ and $p_6$ are determined by $p_1,p_3,p_4$, and
$p_4$ is the monic quotient of $p_1p_3$ by $p_2=U+t_{2,1}$.
The remainder is $(t_{1,1}-t_{2,1})p_3(-t_{2,1})$.
Thus $\Spec B_0$ is reduced and has exactly two irreducible components,
\[
 \mathcal C_1=V(t_{1,1}-t_{2,1}),\qquad
 \mathcal C_2=V(t_{2,1}^2-t_{3,1}t_{2,1}+t_{3,2}).
\]
Both components are three-dimensional.

Recall that the slice is $P=[-\tfrac12,\tfrac32]$. Put $E=[0,1]$,
$L=[-\tfrac12,0]$, and $R=[0,\tfrac12]$. Its two maximal
lattice-friendly decompositions, up to equivalence, are
\begin{equation}\label{eq:eighth-decompositions}
 \mathcal D_1:\quad P=\{-\tfrac12\}+E+E,\qquad
 \mathcal D_2:\quad P=L+R+E.
\end{equation}
Proposition~\ref{prop:decomposition-family} associates to
$\mathcal D_1$ the polynomials
\[
 p_1=p_2=U-\lambda_0,\qquad
 p_3=p_4=(U-\lambda_1)(U-\lambda_2),
\]
where $\lambda_0$ belongs to the singleton and $\lambda_1,\lambda_2$
to the two unit intervals. Comparing coefficients with
$p_1=U+t_{1,1}$, $p_2=U+t_{2,1}$ and
$p_3=U^2+t_{3,1}U+t_{3,2}$ defines
$\psi_{\mathcal D_1}:\mathbb A^3\to\Spec B_0$ by
\[
 t_{1,1}=t_{2,1}=-\lambda_0,\qquad
 t_{3,1}=-(\lambda_1+\lambda_2),\qquad
 t_{3,2}=\lambda_1\lambda_2.
\]
Its image is exactly $\mathcal C_1$.

For $\mathcal D_2$, assign $\lambda_L,\lambda_R,\lambda_E$
to $L,R,E$. The proposition gives
\[
\begin{aligned}
 p_1&=U-\lambda_L,&p_2&=U-\lambda_R,\\
 p_3&=(U-\lambda_R)(U-\lambda_E),&
 p_4&=(U-\lambda_L)(U-\lambda_E).
\end{aligned}
\]
Comparing coefficients defines
$\psi_{\mathcal D_2}:\mathbb A^3\to\Spec B_0$ by
\[
 t_{1,1}=-\lambda_L,\quad t_{2,1}=-\lambda_R,\quad
 t_{3,1}=-(\lambda_R+\lambda_E),\quad
 t_{3,2}=\lambda_R\lambda_E.
\]
These formulas satisfy the defining equation of $\mathcal C_2$.
Conversely, every point of $\mathcal C_2$ is obtained by setting
\[
 \lambda_L=-t_{1,1},\qquad
 \lambda_R=-t_{2,1},\qquad
 \lambda_E=t_{2,1}-t_{3,1},
\]
since $\lambda_R\lambda_E=t_{2,1}t_{3,1}-t_{2,1}^2=t_{3,2}$
on $\mathcal C_2$. Thus $\psi_{\mathcal D_2}$ is an isomorphism
onto $\mathcal C_2$.

Setting $t_{2,1}=0$, as in Example~\ref{ex:eighth-generator-change},
gives the universal equivariant base for $X$ along the chosen
degree ray $\Z_{>0}m$,
\[
 B\simeq\C[t_{1,1},t_{3,1},t_{3,2}]/(t_{1,1}t_{3,2}),
\]
whose two components are affine planes, as in
Corollary~\ref{cor:coefficient-components}. The parameters have
degrees $m,m,2m$, and the relation has no linear term. Hence
\[
 \dim T_X^1(-m)=2,\qquad \dim T_X^1(-2m)=1,
 \qquad T_X^1(-jm)=0\quad(j\ge3).
\]

This non-Gorenstein example lies outside the hypotheses of
\cite[Theorem 6.2]{ACF}: those hypotheses imply
$T_X^1(-jm)=0$ for $j\ge2$, by \cite[Remark 6.4]{ACF}.
\end{example}

\section{Arbitrary primitive degrees}\label{sec:arbitrary-degrees}\label{subsec:truncation}

In this section, we extend equivariant universality and the component
correspondence to arbitrary primitive nonzero $m\in\LL$ by
applying the pair construction to a truncation of the cone.

For $m\notin -S$, set
\[
 \tau=\sigma\cap[m\ge0],\qquad
 S_\gamma=\gamma^\vee\cap\LL,\qquad
 A_\gamma=\C[S_\gamma],\qquad X_\gamma=\Spec A_\gamma
 \quad(\gamma=\sigma,\tau).
\]
Both cones are full-dimensional and strongly convex, and $m\in S_\tau$.
Their slices at height one coincide. The semigroup inclusion
$S_\sigma\subseteq S_\tau$ induces a toric birational morphism
$X_\tau\to X_\sigma$. Altmann transfers the deformation families
associated with Minkowski decompositions of $P_m$ from $X_\tau$
to $X_\sigma$ over the same base, compatibly with this morphism
\cite[Theorem 3.4]{AltmannFamilies}.
In this section, we show that restriction identifies the equivariant
deformation functors.

The semigroups satisfy
\begin{equation}\label{eq:truncationdownward}
 s\in S_\sigma,\quad s-jm\in S_\tau,\quad j\ge0
 \quad\Longrightarrow\quad s-jm\in S_\sigma.
\end{equation}
Indeed, the inequalities on rays where $m\ge0$ hold in $S_\tau$;
on the remaining rays, subtracting $jm$ increases the pairing.

For a graded flat deformation $A_{\tau,R}$, choose homogeneous
monomial lifts $e_s$, which form an $R$-basis by
Lemma~\ref{lem:basis}. By \eqref{eq:truncationdownward},
\begin{equation}\label{eq:truncationsubalgebra}
 \mathcal R(A_{\tau,R})
 :=\sum_{s\in S_\sigma}R(A_{\tau,R})_s
 =\bigoplus_{s\in S_\sigma}Re_s\subset A_{\tau,R}.
\end{equation}
This subalgebra is free over $R$, with special fibre $A_\sigma$.
It is finitely generated over $R$: lifts of finitely many
semigroup generators generate modulo $\mathfrak m_R$, and then
generate the algebra by nilpotence of $\mathfrak m_R$.
Its intrinsic definition and basis description show compatibility
with isomorphisms and base change, giving
$\mathcal R:\Def_m^{X_\tau}\longrightarrow\Def_m^{X_\sigma}$.

Altmann \cite[Theorem~2.5 and Section~3.5]{AltmannFamilies}
showed that truncation preserves the homogeneous tangent space
in the chosen degree. Since $\sigma\cap[jm\ge0]=\tau$ for every
$j\ge1$, his result gives
\[
 T_{X_\tau}^1(-jm)\simeq T_{X_\sigma}^1(-jm),\qquad j\ge1.
\]
For background on the Harrison complex and its relation to
Andr\'e--Quillen and Hochschild cohomology in the toric setting,
see \cite{FilipHochschild}.
The next lemma establishes the corresponding isomorphisms
$T_{X_\tau}^2(-jm)\simeq T_{X_\sigma}^2(-jm)$ and proves
compatibility with obstruction classes.

\begin{lemma}\label{lem:truncationcohomology}
For every $j\ge1$, restriction of homogeneous Harrison cochains induces
isomorphisms
\begin{equation}\label{eq:truncationcohomology}
 T_{X_\tau}^i(-jm)\xrightarrow{\ \sim\ }T_{X_\sigma}^i(-jm),
 \qquad i=1,2.
\end{equation}
These are the tangent and obstruction maps associated to $\mathcal R$.
\end{lemma}
\begin{proof}
The case $\tau=\sigma$ is immediate. Assume $\tau\ne\sigma$
and fix $j\ge1$. We first compare the quasilinear complexes
for $\sigma$ and $\tau$. We then identify this
comparison with the restriction of Harrison cochains and check
its compatibility with lifting deformations.

For a ray $\rho$ of either cone, let $v_\rho\in\LL^*$ denote
its primitive lattice generator. For $\gamma=\sigma,\tau$ and
$F\preceq\gamma$, put
\[
 K_{\gamma,F}^{jm}
 =\{a\in S_\gamma\setminus\{0\}:
   \langle a,v_\rho\rangle<j\langle m,v_\rho\rangle
   \text{ for every ray }\rho\preceq F\}.
\]
For $F=0$, this is $S_\gamma\setminus\{0\}$.
Let $\operatorname{Hom}_{\mathrm{ql}}(K,\C)$ denote the functions
$f:K\to\C$ satisfying $f(a+b)=f(a)+f(b)$ whenever
$a,b,a+b\in K$. By \cite[Theorem~5.3]{AltmannSletsjoe},
which applies to arbitrary rational polyhedral cones,
$T_{X_\gamma}^i(-jm)$ is the $i$th cohomology, for $i=1,2$, of
\[
 C_\gamma^p(jm)=
 \bigoplus_{\substack{F\preceq\gamma\\\dim F=p}}
 \operatorname{Hom}_{\mathrm{ql}}(K_{\gamma,F}^{jm},\C),
\]
with the signed face-restriction differentials.

We will show that restriction induces an isomorphism
\begin{equation}\label{eq:quasilinear-restriction}
 C_\tau^\bullet(jm)\xrightarrow{\ \sim\ }C_\sigma^\bullet(jm).
\end{equation}
A nonzero face has $K_{\gamma,F}^{jm}=\varnothing$ unless
$m$ is positive on all its
rays. Such faces are common to $\sigma$ and $\tau$: every new
or truncated nonzero face of $\tau$ contains a ray in $[m=0]$,
whereas a face of $\sigma$ not retained by the cut has a ray
on which $m$ is negative.
For each contributing common face $F$, including $F=0$, choose
a nonzero lattice vector
$w\in\LL\cap\operatorname{relint}(\sigma^\vee\cap F^\perp)$.
The face $F$ is proper in $\sigma$, so this is possible.
Then $w$ vanishes on $F$ and is positive on every other ray
of $\sigma$. Consequently, for $a\in K_{\tau,F}^{jm}$,
we have $a+nw\in K_{\sigma,F}^{jm}$ for $n\gg0$,
and every positive multiple of $w$ lies in $K_{\sigma,F}^{jm}$.
Given a quasilinear function $f$ on $K_{\sigma,F}^{jm}$, define
\[
 \widetilde f(a)=f(a+nw)-nf(w),\qquad n\gg0.
\]
The equality $f(a+(n+1)w)=f(a+nw)+f(w)$ shows that this
does not depend on the sufficiently large choice of $n$.
It agrees with $f$ on $K_{\sigma,F}^{jm}$, and any quasilinear
extension must satisfy this formula. For $a,b,a+b\in K_{\tau,F}^{jm}$,
choose $n$ sufficiently large that all shifted arguments below belong
to $K_{\sigma,F}^{jm}$. Using the same $w$, quasilinearity of $f$ gives
\[
\begin{aligned}
 \widetilde f(a)+\widetilde f(b)
 &=f(a+nw)+f(b+nw)-2nf(w)\\
 &=f(a+b+2nw)-2nf(w)\\
 &=\widetilde f(a+b).
\end{aligned}
\]
Here the last equality uses the defining formula with shift $2n$.
Thus restriction is an isomorphism on every contributing face.
With matching face orientations, these maps commute with the
differentials, proving \eqref{eq:quasilinear-restriction}.

We next show that the induced cohomology maps are obtained by
restricting Harrison cochains on $A_\tau$ to inputs in the
subalgebra $A_\sigma$. For a $q$-cochain $h$ of degree $-jm$,
this is the map $h|_{A_\sigma^{\otimes q}}$, whose values lie
in $A_\sigma$ by \eqref{eq:truncationdownward}.
This identification will let us compare the obstruction classes.
Put $S_\gamma^+=S_\gamma\setminus\{0\}$ and let
$\mathcal H_\gamma^q$ be the normalized Harrison cochains of
degree $-jm$ on $A_\gamma$. Writing their monomial coefficients as
\[
 h(\chi^{a_1},\ldots,\chi^{a_q})
 =f(a_1,\ldots,a_q)\chi^{a_1+\cdots+a_q-jm}
\]
identifies them with scalar Harrison cochains on $(S_\gamma^+)^q$
subject to
\[
 f(a_1,\ldots,a_q)=0
 \quad\text{if}\quad a_1+\cdots+a_q-jm\notin S_\gamma.
\]
This is the relative complex of
\cite[Section~2.2 and Theorem~2.3]{AltmannSletsjoe}.

To relate these cochains to the quasilinear complex, let
$\mathcal C^q(K)$ denote scalar Harrison cochains on
$q$-tuples in $K$ whose sum is in $K$, and form the double complex
\[
 D_\gamma^{p,q}
 =\bigoplus_{\substack{F\preceq\gamma\\\dim F=p}}
    \mathcal C^q(K_{\gamma,F}^{jm}),
 \qquad p\ge0,\quad q\ge1.
\]
Its vertical differential $\delta$ is the inhomogeneous Harrison
differential, and its horizontal differential $d$ is signed
face restriction; the total differential is $d+(-1)^p\delta$.
The coefficient inclusion
$\alpha_\gamma:\mathcal H_\gamma^\bullet\to
\operatorname{Tot}D_\gamma$ places $f$ in $D_\gamma^{0,q}$
and zero elsewhere. Its horizontal differential vanishes:
a tuple whose sum lies in $K_{\gamma,\rho}^{jm}$ for a ray
$\rho$ satisfies the displayed vanishing condition.
The augmented horizontal rows are exact by
\cite[Lemma~3.2]{AltmannSletsjoe}, so $\alpha_\gamma$ is a
quasi-isomorphism.

Restriction of tuple functions on each common face defines
$r:D_\tau\to D_\sigma$; the remaining face summands are zero.
It commutes with $\delta$ and $d$, since both differentials
use only addition of inputs and restriction to faces.
Restricting inputs gives the cochain map $\operatorname{res}$
in the commutative square
\[
\begin{array}{ccc}
 \mathcal H_\tau^\bullet&
 \xrightarrow{\alpha_\tau}&\operatorname{Tot}D_\tau\\
 {\scriptstyle\operatorname{res}}\downarrow&&\downarrow{\scriptstyle r}\\
 \mathcal H_\sigma^\bullet&
 \xrightarrow{\alpha_\sigma}&\operatorname{Tot}D_\sigma.
\end{array}
\]
Indeed, both composites restrict the same coefficient function
to $(S_\sigma^+)^q$ in the zero-face column.

Taking vertical cohomology gives
$E_1^{p,q}=H_\delta^q(D_\gamma^{p,\bullet})$.
Its $q=1$ row is $C_\gamma^\bullet(jm)$, and the induced map
is precisely the quasilinear restriction already considered.
Moreover, $E_1^{0,q}=E_1^{1,q}=0$ for $q\ge2$ by
\cite[Corollary~4.4]{AltmannSletsjoe}.
In total degrees $2$ and $3$, these vanishings leave only the
$q=1$ row. Its inclusion therefore induces
\[
 H^i(C_\gamma^\bullet(jm))
 \xrightarrow{\sim}H^{i+1}(\operatorname{Tot}D_\gamma),
 \qquad i=1,2,
\]
as in \cite[Theorem~5.3]{AltmannSletsjoe}.
These inclusions commute with $r$, so the maps
$T_{X_\tau}^i(-jm)\to T_{X_\sigma}^i(-jm)$, $i=1,2$,
induced by \eqref{eq:quasilinear-restriction} agree with those
induced by restricting Harrison cochains. This proves
\eqref{eq:truncationcohomology}.

It remains to check that these are the maps used in lifting
deformations. Let $R'\twoheadrightarrow R$ be a small graded extension
with kernel $E$. Lift the multiplication of $A_{\tau,R}$ to a
homogeneous commutative unital operation on
$\bigoplus_{s\in S_\tau}R'e_s$, keeping $e_0=1$ as the unit.
By \eqref{eq:truncationdownward}, its restriction preserves
$\bigoplus_{s\in S_\sigma}R'e_s$, and its associator restricts
to the associator of this restricted operation.
Here the associator measures the failure of the chosen operation
to be associative.
Since $\mathfrak m_{R'}E=0$, these associators represent the
obstruction classes in $(T_{X_\tau}^2\otimes_\C E)_0$ and
$(T_{X_\sigma}^2\otimes_\C E)_0$, respectively.
Likewise, differences of associative liftings and the
coboundaries arising from changes of basis restrict.
Thus the maps
\[
 (T_{X_\tau}^i\otimes_\C E)_0
 \longrightarrow (T_{X_\sigma}^i\otimes_\C E)_0,
 \qquad i=1,2,
\]
carry obstruction classes for $i=2$ and difference classes for
$i=1$ to those of the restricted deformations.
\end{proof}

\begin{theorem}\label{thm:truncation}
The transformation $\mathcal R$ is an isomorphism of equivariant
deformation functors. In particular, restricting a formally
universal $\tau$-family gives a formally universal $\sigma$-family
over the same base.
\end{theorem}
\begin{proof}
We first prove formal smoothness: a prescribed lifting of the
restricted deformation comes from a lifting of the original
deformation.
Let $R'\twoheadrightarrow R$ be a small graded extension with
kernel $E$. Suppose we are given $A_{\tau,R}$ and a lifting
$A_{\sigma,R'}$ with a fixed identification
$A_{\sigma,R'}\otimes_{R'}R\simeq\mathcal R(A_{\tau,R})$.
We must find a lifting of $A_{\tau,R}$ whose restriction is
isomorphic to $A_{\sigma,R'}$ compatibly with this identification.

The obstruction to lifting $A_{\tau,R}$ over $R'$ lies in
$(T_{X_\tau}^2\otimes_\C E)_0$.
By Lemma~\ref{lem:truncationcohomology}, its image in
$(T_{X_\sigma}^2\otimes_\C E)_0$ is the obstruction to lifting
$\mathcal R(A_{\tau,R})$ over $R'$.
This image is zero, since the given $A_{\sigma,R'}$ already
provides such a lifting.
Injectivity of this map therefore makes the original obstruction
vanish, so a lifting $A_{\tau,R'}$ exists by the standard lifting theorem
\cite[Theorem~10.1]{HartshorneDef}.
Here total degree includes the grading of $E$: a cochain of
degree $-jm$ with coefficient in $E_{jm}$ has total degree zero.
Since the Harrison differential preserves total degree, the
cochain used to cancel the obstruction may be replaced by its
degree-zero part. Thus the lifting can be chosen graded.

This lifting need not yet have the prescribed restriction.
We therefore compare $\mathcal R(A_{\tau,R'})$ with
$A_{\sigma,R'}$, keeping their identifications over $R$ fixed.
By the same lifting theorem, their difference is a class in
$(T_{X_\sigma}^1\otimes_\C E)_0$, and such classes describe
all changes of a lifting up to isomorphism respecting its reduction.
The isomorphism
$(T_{X_\tau}^1\otimes_\C E)_0\xrightarrow{\sim}
(T_{X_\sigma}^1\otimes_\C E)_0$
induced by Lemma~\ref{lem:truncationcohomology} lifts this difference
to $(T_{X_\tau}^1\otimes_\C E)_0$.
Changing $A_{\tau,R'}$ by that class gives the required restriction,
by the compatibility proved in the lemma. This proves formal smoothness.

Formal smoothness gives surjectivity of $\mathcal R$ over every
Artinian base, by induction over small extensions. It remains
to prove injectivity. If $P_m$ is an integral translate of its
tail cone, its rounded support function is integral linear.
Thus all $\eta_{\Z,\beta}$ vanish and $B_0=B=\C$.
The source is then the point functor by
Corollary~\ref{cor:universal-X}, and surjectivity makes the
target the point functor as well.

In the remaining case, we first show that a $\sigma$-deformation
has no nonidentity equivariant automorphisms preserving its
special fibre. This will make the identifications over $R$
agree in the final lifting argument. We claim that
\begin{equation}\label{eq:truncation-derivations}
 \operatorname{Der}_{\C}(\C[S_\sigma])_{-jm}=0
 \qquad\text{for all }j\ge1.
\end{equation}

A derivation of $\C[S_\sigma]$ of degree $-jm$ extends to the
torus as $\chi^{-jm}\partial_n$, where
$\partial_n(\chi^s)=\langle s,n\rangle\chi^s$ and
$n\in\operatorname{Hom}_{\Z}(\LL,\C)$.
For each primitive ray vector $v$ with $\langle m,v\rangle>0$,
regularity on $S_\sigma\cap v^\perp$ forces $n\in\C v$.
Indeed, subtracting $jm$ gives negative pairing with $v$,
so $n$ must vanish on this facet of the semigroup.
There is at least one such ray because $m\notin-S$;
two such rays force $n=0$. If there is only one, choose
$s\in S_\sigma$ with $\langle s,v\rangle=1$: primitiveness
provides a lattice element of height one, and adding a sufficiently
large lattice element in the relative interior of
$\sigma^\vee\cap v^\perp$ puts it in $S_\sigma$.
If $n\ne0$, regularity on this monomial requires
$1-j\langle m,v\rangle\ge0$, hence
$j=\langle m,v\rangle=1$.
In that case $P_m$ has the single vertex $v$ and tail cone
$\sigma\cap[m=0]$, so it is the excluded integral translate.
This proves \eqref{eq:truncation-derivations}.
An equivariant automorphism preserving the special fibre is
the identity modulo the base maximal ideal. Its first nonzero
term in the filtration by powers of that ideal would be a
nonzero element of
$(\operatorname{Der}_{\C}(\C[S_\sigma])\otimes_\C
\mathfrak m_R^a/\mathfrak m_R^{a+1})_0$ for some $a\ge1$.
Every coefficient has degree $jm$ with $j\ge1$, so
\eqref{eq:truncation-derivations} rules this out.

We can now prove injectivity by induction over small extensions.
Suppose two $\tau$-deformations over
$R'\twoheadrightarrow R$ have isomorphic restrictions.
Inductively their reductions over $R$ are isomorphic; fix such
an identification. Its restriction agrees with the reduction
of the given isomorphism, since the two differ by an equivariant
automorphism of the common $\sigma$-deformation, which must be
the identity.
The difference between the two liftings is consequently a class
in $(T_{X_\tau}^1\otimes_\C E)_0$ mapping to zero for $\sigma$.
Injectivity in Lemma~\ref{lem:truncationcohomology} makes this
class zero, so the liftings are isomorphic. Induction proves
injectivity of $\mathcal R$, and hence the assertion.
\end{proof}

\begin{remark}\label{rem:nonpositive}\label{subsec:nonpositive}
If $m\in -S$, then $P_m=\varnothing$, and every nonzero-face term
in the same cohomology complex vanishes. Hence $T_X^1(-jm)=0$
for all $j\ge1$. Induction over small
extensions shows that every graded deformation is trivial.
Thus $\Def_m^X$ is the point functor, prorepresented by $B=\C$.
\end{remark}

\begin{proof}[Proof of Theorem~\ref{thm:main}]
We construct the restricted algebraic family, prove
its finite generation, and transfer universality and the component
description from $X_\tau$.
Apply Section~\ref{subsec:minimal} to $X_\tau$.
Corollary~\ref{cor:universal-X} gives a family $\mathcal A_\tau/B$
with finitely generated positively graded base $B$ whose completion
is formally universal for $\Def_m^{X_\tau}$. In the integral-translate
case, $B=\C$.
Using its homogeneous monomial basis $(e_s)_{s\in S_\tau}$,
the restriction formula \eqref{eq:truncationsubalgebra} gives
\[
 \mathcal A_\sigma=\bigoplus_{s\in S_\sigma}Be_s
 \subset\mathcal A_\tau.
\]
By \eqref{eq:truncationdownward}, multiplication preserves this
submodule. It is therefore a flat $B$-algebra with special fibre
$A_\sigma$, and its construction commutes with base change.

It remains to check finite generation over $B$.
Choose an integral $\ell\in\operatorname{int}\sigma$ with
$\ell(m)>0$, which exists because $m\notin-S$.
Let $s_1,\ldots,s_h$ generate $S_\sigma$ as a semigroup.
For $s=\sum_\nu a_\nu s_\nu$, homogeneity and the special-fibre
multiplication give
\[
 \prod_{\nu=1}^h e_{s_\nu}^{a_\nu}
 =e_s+\sum_{j>0}b_j e_{s-jm},\qquad b_j\in B_{jm},
\]
with finitely many nonzero terms. Their exponents lie in
$S_\sigma$ by \eqref{eq:truncationdownward}, and
$\ell(s-jm)<\ell(s)$. Induction on the nonnegative integer
$\ell(s)$ therefore expresses every $e_s$ in terms of
$e_{s_1},\ldots,e_{s_h}$. These lifts generate
$\mathcal A_\sigma$ over $B$.

By base-change compatibility, Theorem~\ref{thm:truncation}
proves formal universality of its completion for $\Def_m^{X_\sigma}$.
Since $\sigma$ and $\tau$ have the same slice $P_m$,
Proposition~\ref{prop:componentimages} and
Corollary~\ref{cor:coefficient-components} give the component
correspondence and dimension $k-1$ for $k>0$ nontrivial summands.
The family on each component is obtained by base change;
its completion gives the formal component family.
\end{proof}

\begin{corollary}\label{cor:ordinary-arbitrary}
Let $m\in\LL$ be primitive. If $\Def^{X_\sigma}$ admits a hull
with ring $H_{X_\sigma}$, chosen torus-equivariantly as in
Subsection~\ref{subsec:usual-miniversal}, then
$\widehat B\simeq H_{X_\sigma,m}$, with the notation of
\eqref{eq:ordinary-ray-base}. The completed family constructed
above is the restriction of the corresponding miniversal formal
family to the parameter degrees $jm$, $j\ge1$.

For $m\notin-S$, put $\tau=\sigma\cap[m\ge0]$.
If $\Def^{(X_\tau,V(\chi^m))}$ admits a hull, choose its
miniversal formal family torus-equivariantly as in
Subsection~\ref{subsec:usual-miniversal}. Its restriction to the
parameter degrees $jm$, $j\ge1$, is isomorphic to
$\widehat\pi_0$ over $\widehat B_0$.
\end{corollary}
\begin{proof}
Theorem~\ref{thm:main} and
Remark~\ref{rem:nonpositive} show that $\widehat B$ prorepresents
$\Def_m^{X_\sigma}$. The comparison argument of
Corollary~\ref{cor:ordinary-X} therefore identifies this base
and family with the restriction of the chosen miniversal
formal family over $H_{X_\sigma}$.
For the pair, $m\in S_\tau$, so
Corollary~\ref{cor:ordinary-pair} applies directly to $X_\tau$.
\end{proof}

Together with Corollary~\ref{cor:ordinary-pair}, this proves
Theorem~\ref{thm:usual-miniversality}.

\begin{example}\label{ex:mixed-sign}
Let $X_\sigma=X_0\times\mathbb A^1_v$, where
$X_0=\frac18(1,3)$ and $m_0=(2,2)$ are as in
Example~\ref{ex:eighth}. Set
\[
 \LL_0=\{(p,q)\in\Z^2:p+3q\in8\Z\},\qquad
 \LL=\LL_0\oplus\Z,\qquad m=(m_0,-1)\notin S_\sigma.
\]
Write $Z=\chi^{(m_0,0)}$ for the character called $U$ in that
example, and $v=\chi^{(0,1)}$ for the last coordinate.
With the notation of Example~\ref{ex:eighth}, put
\[
 Q_\beta(Z,v)=Z^{\eta_{\Z,\beta}}
   +\sum_{j=1}^{\eta_{\Z,\beta}}
       t_{\beta,j}Z^{\eta_{\Z,\beta}-j}v^j,
 \qquad \deg t_{\beta,j}=jm.
\]
The deformed equations are
$G_\beta=\LM(g_\beta)-Q_\beta(Z,v)W_\beta$.
Their overlap conditions are
\[
 Q_5=Q_1Q_3=Q_2Q_4,\qquad Q_6=Q_3Q_4.
\]
Comparing coefficients gives the same coefficient base $B_0$ as in
Example~\ref{ex:eighth-components}. The homogeneous change
$Z'=Z+t_{2,1}v$ normalizes $t_{2,1}=0$.
Renaming $Z'$ as $Z$ and writing $a=t_{1,1}$, $b=t_{3,1}$,
$c=t_{3,2}$ for the normalized coefficients gives
\[
 B=\C[a,b,c]/(ac),\qquad
 \deg a=\deg b=m,\quad\deg c=2m.
\]
The normalized polynomials are
\[
\begin{aligned}
 q_1&=Z+av,&q_2&=Z,\\
 q_3&=Z^2+bZv+cv^2,
 &q_4&=Z^2+(a+b)Zv+(c+ab)v^2,\\
 q_5&=q_1q_3=q_2q_4,&q_6&=q_3q_4.
\end{aligned}
\]
Thus the equivariant family over $B$ has equations
\begin{equation}\label{eq:mixed-sign-sigma}
\begin{aligned}
 z_1^2&=q_1z_0,&z_3^2&=q_2z_4,\\
 z_0z_3&=q_3z_1,&z_1z_4&=q_4z_3,\\
 z_1z_3&=q_5,&z_0z_4&=q_6.
\end{aligned}
\end{equation}
The standard monomials $Z^qv^kW$, with $q,k\ge0$ and
$W\in\mathcal W$, form a free $B$-basis.
Since $Q_\beta(Uv,v)=v^{\eta_{\Z,\beta}}p_\beta(U)$,
this is the restriction under $Z=Uv$ of the family for
$\tau=\sigma\cap[m\ge0]$, where $U=\chi^m$; its completion
is universal by Theorem~\ref{thm:truncation}.
Setting $v=1$ recovers exactly the normalized family of
Example~\ref{ex:eighth-components} over the same base $B$.
\end{example}


\begin{thebibliography}{ACF22a}
\small
\setlength{\baselineskip}{11pt}
\setlength{\itemsep}{0pt}
\setlength{\parsep}{0pt}
\setlength{\parskip}{0pt}
\bibitem[AL94]{AL}
W.~W.~Adams and P.~Loustaunau,
\emph{An Introduction to Gr\"obner Bases},
Graduate Studies in Mathematics \textbf{3},
American Mathematical Society, Providence, RI, 1994.

\bibitem[Alt97]{Altmann}
K.~Altmann,
\emph{The versal deformation of an isolated, toric Gorenstein singularity},
Inventiones Mathematicae \textbf{128} (1997), 443--479.

\bibitem[Alt00]{AltmannFamilies}
K.~Altmann,
\emph{One-parameter families containing three-dimensional toric Gorenstein singularities},
in \emph{Explicit Birational Geometry of 3-Folds} (A.~Corti and M.~Reid, eds.),
London Mathematical Society Lecture Note Series 281, Cambridge University Press, 2000, 21--50.

\bibitem[ACF22a]{ACFpoly}
K.~Altmann, A.~Constantinescu, and M.~Filip,
\emph{Polyhedra, lattice structures, and extensions of semigroups},
Journal of the London Mathematical Society \textbf{106} (2022), 3938--4008.

\bibitem[ACF22b]{ACF}
K.~Altmann, A.~Constantinescu, and M.~Filip,
\emph{Versality in toric geometry},
Journal of Algebra \textbf{609} (2022), 1--43.

\bibitem[AK13]{AltmannKastner}
K.~Altmann and L.~Kastner,
\emph{Negative deformations of toric singularities that are smooth in codimension two},
in \emph{Deformations of Surface Singularities}
(A.~N\'emethi and \'A.~Szil\'ard, eds.),
Bolyai Society Mathematical Studies \textbf{23},
Springer, Berlin--Heidelberg, 2013, 13--55.

\bibitem[AS98]{AltmannSletsjoe}
K.~Altmann and A.~B.~Sletsj\o e,
\emph{Andr\'e--Quillen cohomology of monoid algebras},
Journal of Algebra \textbf{210} (1998), 708--718.

\bibitem[Art74]{ArtinVersal}
M.~Artin,
\emph{Versal deformations and algebraic stacks},
Inventiones Mathematicae \textbf{27} (1974), 165--189.

\bibitem[CFP22]{CFP}
A.~Corti, M.~Filip and A.~Petracci,
\emph{Mirror symmetry and smoothing Gorenstein toric affine 3-folds},
in \emph{Facets of Algebraic Geometry}, Vol.~I,
London Mathematical Society Lecture Note Series \textbf{472},
Cambridge University Press, Cambridge, 2022, 132--163.

\bibitem[Fil18]{FilipHochschild}
M.~Filip,
\emph{Hochschild cohomology and deformation quantization of affine toric varieties},
Journal of Algebra \textbf{506} (2018), 188--214.

\bibitem[Fil26a]{FilipPair}
M.~Filip,
\emph{Deformations of an affine Gorenstein toric pair},
Journal of Algebra \textbf{687} (2026), 419--445.

\bibitem[Fil26b]{FilipLaurent}
M.~Filip,
\emph{Laurent polynomials and deformations of non-isolated Gorenstein toric singularities},
Advances in Mathematics \textbf{487} (2026), Article 110775.

\bibitem[Gra72]{Grauert}
H.~Grauert,
\emph{\"Uber die Deformation isolierter Singularit\"aten analytischer Mengen},
Inventiones Mathematicae \textbf{15} (1972), 171--198.

\bibitem[Har10]{HartshorneDef}
R.~Hartshorne,
\emph{Deformation Theory},
Graduate Texts in Mathematics \textbf{257},
Springer, New York, 2010.

\bibitem[IV12]{IltenVollmert}
N.~O.~Ilten and R.~Vollmert,
\emph{Deformations of rational $T$-varieties},
Journal of Algebraic Geometry \textbf{21} (2012), no.~3, 531--562.

\bibitem[JP00]{deJongPfister}
T.~de Jong and G.~Pfister,
\emph{Local Analytic Geometry: Basic Theory and Applications},
Advanced Lectures in Mathematics,
Vieweg, Braunschweig--Wiesbaden, 2000.

\bibitem[JS94]{deJongVanStraten}
T.~de Jong and D.~van Straten,
\emph{On the deformation theory of rational surface singularities
with reduced fundamental cycle},
Journal of Algebraic Geometry \textbf{3} (1994), no.~1, 117--172.

\bibitem[KS88]{KSB}
J.~Koll\'ar and N.~I.~Shepherd-Barron,
\emph{Threefolds and deformations of surface singularities},
Inventiones Mathematicae \textbf{91} (1988), no.~2, 299--338.

\bibitem[Rim80]{Rim}
D.~S.~Rim,
\emph{Equivariant $G$-structure on versal deformations},
Transactions of the American Mathematical Society
\textbf{257} (1980), no.~1, 217--226.

\bibitem[Ser06]{Sernesi}
E.~Sernesi,
\emph{Deformations of Algebraic Schemes},
Grundlehren der mathematischen Wissenschaften \textbf{334},
Springer-Verlag, Berlin--Heidelberg, 2006.

\bibitem[Stacks]{Stacks}
The Stacks Project Authors,
\emph{The Stacks Project},
\href{https://stacks.math.columbia.edu/tag/06G8}{Tag 06G8} (hulls),
\href{https://stacks.math.columbia.edu/tag/051F}{Tag 051F} (lifting bases), and
\href{https://stacks.math.columbia.edu/tag/00MB}{Tag 00MB} (exactness of completion).

\bibitem[Ste91]{StevensCyclic}
J.~Stevens,
\emph{On the versal deformation of cyclic quotient singularities},
in \emph{Singularity Theory and its Applications, Warwick 1989,
Part I: Geometric Aspects of Singularities},
Lecture Notes in Mathematics \textbf{1462},
Springer-Verlag, Berlin--Heidelberg, 1991, 302--319.

\bibitem[Ste03]{StevensBook}
J.~Stevens,
\emph{Deformations of singularities},
Lecture Notes in Mathematics \textbf{1811},
Springer, Berlin, 2003.

\end{thebibliography}
\end{document}